\documentclass[a4paper,11pt]{article}

\usepackage[english,activeacute]{babel}

\usepackage{xcolor}
\usepackage{lipsum}
\usepackage{float}

\usepackage{needspace}
\usepackage{array}
\newcolumntype{M}[1]{>{\centering\arraybackslash \large}m{#1}}
\newcolumntype{N}[1]{>{\centering\arraybackslash \Large}m{#1}}

\usepackage{amssymb,amsmath,amscd,amsfonts,amsthm,latexsym,bbm}
\usepackage{xcolor,graphpap,fancybox,framed,enumerate,xcolor,mathtools}

\usepackage{enumitem}

\usepackage{graphicx}             

\usepackage[section]{placeins}

\usepackage[right=2.5cm,left=2.5cm,top=2.8cm,bottom=2.8cm]{geometry}

\newtheorem{theorem}{Theorem}[section]
\newtheorem{prop}[theorem]{Proposition}
\newtheorem{coro}[theorem]{Corollary}
\newtheorem{claim}[theorem]{Claim}
\newtheorem{lemma}[theorem]{Lemma}
\newtheorem{defi}[theorem]{Definition}

\renewenvironment{proof}{ \noindent \emph{\textbf{Proof:}}}{\hfill$\square$\\}

\newcommand{\RR}{\mathbb{R}}
\newcommand{\NN}{\mathbb{N}}
\newcommand{\CC}{\mathbb{C}}

\newcommand{\TT}{\mathbb{T}}

\newcommand{\Bc}{\mathcal{B}}
\newcommand{\Cc}{\mathcal{C}}

\newcommand{\Dc}{\mathcal{D}}
\newcommand{\Ec}{\mathcal{E}}
\newcommand{\Hc}{\mathcal{H}}

\newcommand{\Lc}{\mathcal{L}}

\newcommand{\Oc}{\mathcal{O}}

\newcommand{\Qc}{\mathcal{Q}}
\newcommand{\Tc}{\mathcal{T}}

\newcommand{\ddd}{\partial}
\newcommand{\ra}{\rangle}
\newcommand{\la}{\langle}
\newcommand{\grad}{\nabla}
\newcommand{\Dnu}[1]{\frac{\partial #1}{\partial \nu}}

\renewcommand{\div}{\operatorname{div}}
\newcommand{\no}{n$^{\text{o}}$}
\newcommand{\dd}{\,{\text{\rm d}}}

\newcommand{\ii}{{\text{\rm i}}}
\newcommand{\ff}{{\text{\rm f}}}
\newcommand{\id}{{\text{\rm id}}}

\newcommand{\pc}{ \usefont{T1}{cmtl}{m}{n} \selectfont}

\newcommand{\Path}{{\text{\rm Path}}}
\newcommand{\Diff}{{\text{\rm Diff}}}

\newdimen\texpscorrection
\newdimen\figcenter
\def\figurewithtex #1 #2 #3 #4 #5\cr{\null
  {\goodbreak\figcenter=\hsize\relax
  \advance\figcenter by -#4truecm
  \divide\figcenter by 2
  \begin{figure}[hbt]
  \vskip #3truecm\noindent\hskip\figcenter
  \includegraphics{#1}{\hskip\texpscorrection\input #2 }
  \vskip 0.8truecm{\baselineskip=0.8\baselineskip
  \noindent \vbox{\noindent {\footnotesize #5}}\par}
  \end{figure}}}
\def\point#1 #2 #3 {\rlap{\kern #1 truecm
\raise #2 truecm \hbox{#3}}}

\newcounter{compte}
\newenvironment{enum2}{\begin{list}{(\roman{compte})} {\usecounter{compte}
\topsep=1mm \itemsep=0.2mm \leftmargin=10mm  } }
{\end{list}}
\newenvironment{itemize2}{\begin{list}{$\bullet$} {\topsep=1mm \itemsep=0.2mm
\leftmargin=5mm  } }
{\end{list}}

\numberwithin{equation}{section}

\begin{document}

\title{{\bf Control of the Schr\"odinger equation by slow deformations of the domain}}

\author{Alessandro \textsc{Duca}\footnote{Université de Lorraine, CNRS, INRIA, IECL, F-54000 Nancy, France email: {\pc 
alessandro.duca@inria.fr}} ,  Romain 
\textsc{Joly}\footnote{Universit\'e Grenoble Alpes, CNRS, 
Institut Fourier, F-38000 Grenoble, France email: {\pc 
romain.joly@univ-grenoble-alpes.fr}} {~\&~}
Dmitry \textsc{Turaev}
\footnote{Imperial College, London SW7 2AZ, UK, email: {\pc d.turaev@imperial.ac.uk }}}
\date{}

\maketitle


\begin{abstract} 
The aim of this work is to study the controllability of the Schr\"odinger equation
\begin{equation}\label{eq_abstract} 
i\partial_t u(t)=-\Delta u(t)~~~~~\text{ on }\Omega(t)  \tag{$\ast$}
\end{equation}
with Dirichlet boundary conditions, where $\Omega(t)\subset\RR^N$ is a time-varying domain. We prove the global approximate controllability of \eqref{eq_abstract} in $L^2(\Omega)$, via an adiabatic deformation $\Omega(t)\subset\RR^N$ ($t\in[0,T]$) such that $\Omega(0)=\Omega(T)=\Omega$. This control is strongly based on the Hamiltonian structure of \eqref{eq_abstract} provided by \cite{SEmoving}, which enables the use of adiabatic motions. We also discuss several explicit interesting controls that we perform in the specific framework of rectangular domains. 

\vspace{3mm}

\end{abstract}


\section{Introduction}\label{section_intro}

We consider a quantum state confined in a time-varying domain $\{\Omega(t)\}_{t\in I}$ with $I=(0,T)$. Its dynamics is modeled by the following Schr\"odinger equation
\begin{equation}\label{SE}
\left\{\begin{array}{ll} 
i\partial_t u=-\Delta u,~~~~~~&(x,t)\in\Omega(t)\times I,\\
u_{|\partial\Omega(t)}= 0, & (x,t)\in\partial\Omega(t)\times I.
\end{array}\right.
\end{equation}
The aim of this work is to study the controllability of the Schr\"odinger equation \eqref{SE} by considering the time-varying domain $\Omega(t)$ as a control. To be able to consider shapes as rectangular domains, we allow $\Omega(t)$ to admit some corners or edges but no degenerate features as cusps. Let us denote by ``$\Cc^2-$curved polyhedron'' the image of a (non-degenerate) polyhedron via a $\Cc^2-$diffeomorphism. Our main result is as follows.
\begin{theorem}\label{th_main}
Let $d\geq 2$ and $\Omega_0\subset\RR^d$ be a connected open bounded set with $\Cc^2$ boundaries or a $\Cc^2-$curved polyhedron. Let $u_0$ and $u_1$ in $L^2(\Omega_0)$ with $\|u_0\|_{L^2}=\|u_1\|_{L^2}$. For any $\varepsilon>0$, there exist $T>0$ and a smooth family of domains $(\Omega(t))_{t\in [0,T]}$ such that $$\Omega(0)=\Omega(T)=\Omega_0$$ and such that the solution of \eqref{SE} with initial data $u(t=0)=u_0$ satisfies $$\|u(t=T)-u_1\|_{L^2}\leq \varepsilon.$$  
\end{theorem}

Notice that the result of Theorem \ref{th_main} should stay true for more general domains, as soon as the properties of the Dirichlet Laplacian operator in $\Omega_0$ are not too exotic. However, we stick to the above formulation, as it is sufficient for the examples we consider in this paper.

We recall that \eqref{SE} models the evolution of a quantum particle of $\RR^d$ confined by infinite potential walls, for example generated by electric potentials. The above result shows that one can control the quantum state of the particle by changing the shape of the domain enclosed by these walls. 
We emphasize that our process follows a quasi-adiabatic motion and the energy of the particle changes uniformly slowly on the control interval. It provides a new method for driving the system from the ground state to an excited state (or a superposition of excited states), and vice versa, in a soft way, without instantaneous energy changes and without using resonant interactions. The control protocol provided by our proof is ready-to-use in many situations. In the most simple cases, the deformations of the domain are either explicit or based on generic motions, which could be chosen ``randomly''. The main non-explicit parameter is the deformation speed which can be calibrated tentatively in actual/numerical experiments (moving slowly enough in the adiabatic parts or finding a suitable intermediate speed in the non-adiabatic parts).

\vspace{5mm}

\needspace{5mm}
\noindent{\bf \underline{Well-}p\underline{osed unitar}y\underline{ flow for the Schr\"odin}g\underline{er e}q\underline{uation in a movin}g\underline{ domain}}\\[1mm]
The peculiarity of the equation \eqref{SE} is that the phase space $L^2(\Omega(t),\CC)$ depends on time. The existence and uniqueness of solutions for this type of problems was recently studied in \cite{SEmoving}. There, it was shown how to formalize the definition of solutions for the Schr\"odinger equation in time-varying domains by only assuming that the deformation is sufficiently smooth. More precisely, we consider a bounded reference domain $\Omega_0\subset\RR^d$ and a specific family of unitary transformations $h^\sharp(t):L^2(\Omega(t),\CC)\rightarrow L^2(\Omega_0,\CC)$ with inverse $h_\sharp(t)$ such that equation \eqref{SE} is the following equation in $L^2(\Omega_0,\CC)$:
\begin{equation}\label{SE_fixed_intro}
i\partial_t v= h^\sharp(t)H(t)h_\sharp(t) v,~~~~~~(x,t)\in\Omega_0 \times I,
\end{equation}
where the Hamiltonian $H(t)$ is the magnetic Laplacian operator 
$$H(t)=-(\div_x+iA)\circ(\grad_x+iA)-|A|^2$$
with some explicit magnetic potential $A$ depending on the deformation of the domain $\Omega(t)$. More details are recalled in Section \ref{section_moving}.

The new formulation \eqref{SE_fixed_intro} provides a natural framework for the study of the evolution of the Schr\"odinger equation \eqref{SE} and for ensuring the existence and uniqueness of solutions. The Hamiltonian structure of \eqref{SE_fixed_intro} plays a central role in our work as it allows us to use different features of Hamiltonian dynamics, such as the conservation of the $L^2$-norm or the adiabaticity of motion, see Section \ref{section_moving}.

\vspace{5mm}

\needspace{5mm}
\noindent{\bf \underline{Control of the }q\underline{uantum s}y\underline{stem b}y\underline{ deformation of the domain} }\\[1mm]
Our strategy for control is based on specific quasi-adiabatic deformations $(\Omega(t))_{t\in [0,T]}$ of the initial domain $\Omega_0$. Recall that a deformation of $\Omega$ is adiabatic when, for any initial
state with a definite energy, the motion is sufficiently slow so that the system during its evolution stays close to the state defined by the same quantum numbers. It is a well-known fact (the so-called ``avoided level crossing theorem'') that for a typical adiabatic deformation of the domain, if $u_0$ is the ground state in the domain $\Omega(0)$, then the solution $u(t)$ of \eqref{SE} remains close to the ground state of $\Omega(t)$. See Section \ref{sec_adiabatic} for a more precise statement. 
However, we prove Theorem \ref{th_main}  by using a special type of deformations $(\Omega(t))_{t\in [0,T]}$ which drive the system close to energy level crossings and, thus, allow for an adiabatic transition from the ground state to excited states. In our control protocol, the speed of the domain deformation is uniformly slow; we just slightly adjust the speed at the moments near the level crossings in order to distribute the energy between the modes. A typical example is as follows.
\begin{enumerate}[leftmargin=15pt]
\item Start with $u_0$ being the ground state of $\Omega_0$. First, we adiabatically deform $\Omega_0$ into a dumbbell shaped domain $\Omega$: from a smooth part of the boundary of $\Omega_0$, we slowly grow an attached ball $\Omega^L$ linked by a thin channel to the other part, $\Omega^R$, which stays close to the initial shape $\Omega_0$, see Figure \ref{fig-domain}. We do the deformation sufficiently slowly to be adiabatic, so the state $u(t)$, eventually, gets close to the ground state of the dumbbell shaped domain, which, if the channel is sufficiently thin and  the attached ball $\Omega^L$ is sufficiently large, is mostly localized in $\Omega^L$.  
\item At the second step, we adiabatically contract the ball $\Omega^L$. The modes mostly supported by $\Omega^L$ increase their energy during the deformation, while the ones that are mostly localized in $\Omega^R$ stay unaffected. This provides the ``almost crossings'' of the eigenvalues: at certain moments of time, we have two states, one localized mostly in $\Omega^L$ and the other in $\Omega^R$, with sufficiently close energies. From the physical point of view, this allows for a tunneling effect. If we adapt suitably the velocity of the deformation around these critical times, then we can control how much energy is transferred from the modes in $\Omega^L$ to the modes in $\Omega^R$. The main difficulties of the proof of Theorem \ref{th_main} consist in controlling this tunneling effect.
\item Once the desired state has been obtained in $\Omega^R$, we adiabatically deform the domain back to its initial shape $\Omega_0$ by preserving the simplicity of the spectrum. This final phase preserves the distribution of energies obtained at the previous step.
\end{enumerate}
The detailed arguments are provided in Section \ref{section_proof}.

\begin{figure}[ht]
\begin{center}
\resizebox{0.5\textwidth}{!}{
\setlength{\unitlength}{10mm}
\begin{picture}(12,6)
\put(0,0){\includegraphics[width=12cm]{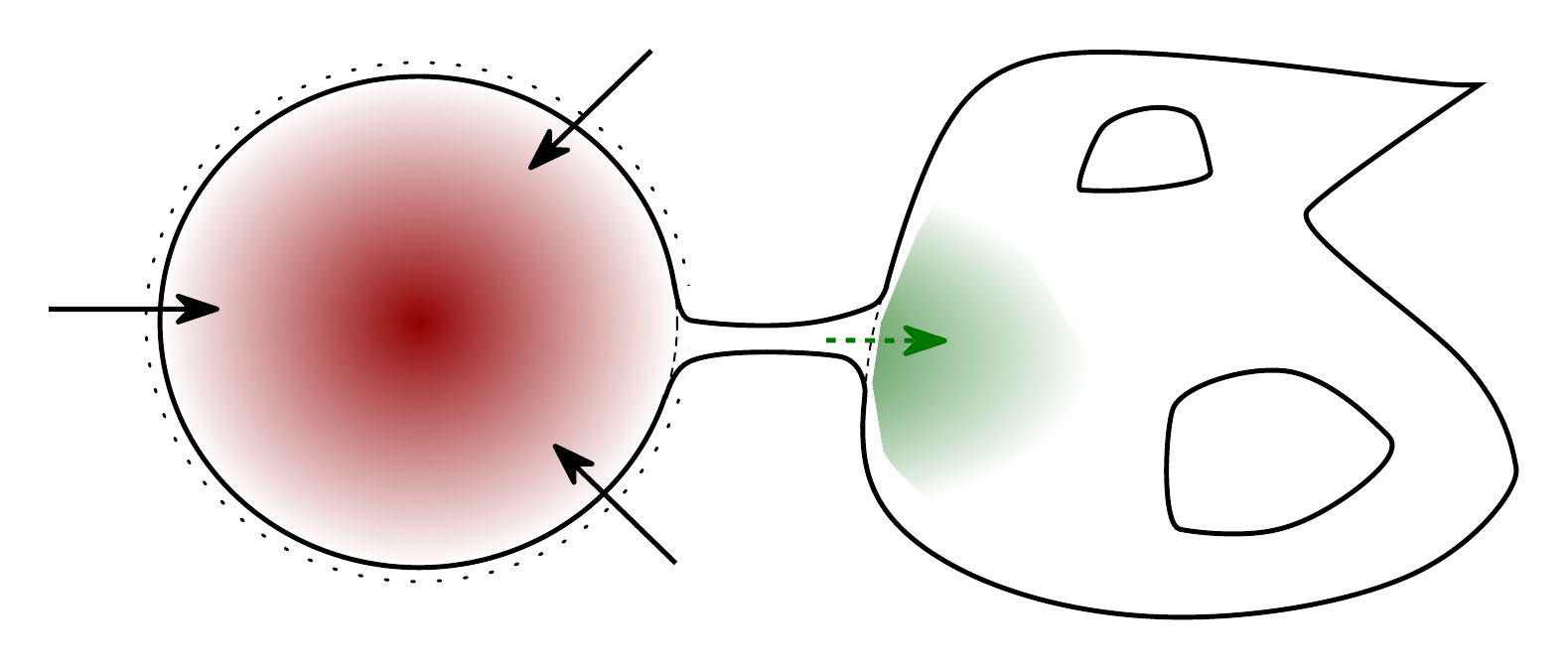}}
\put(0.5,3.8){\resizebox{0.8cm}{!}{$\Omega^L$}}
\put(11,3.3){\resizebox{0.8cm}{!}{$\Omega^R$}}
\put(5.7,3){\resizebox{0.8cm}{!}{$\omega^\eta$}}
\end{picture}}
\end{center}
\caption{\it The key idea is to use a dumbbell shaped domain as pictured here: a ball $\Omega^L$ is linked to a domain $\Omega^R$, close to the original reference domain $\Omega_0$, via a very thin channel $\omega^\eta$. At start, almost all the energy is contained in the left ball $\Omega^L$. When we reduce the size of the ball $\Omega^L$, the energy flows to the right part $\Omega^R$, by the tunneling effect. The technical issue is to control this transfer of energy to create the target state in $\Omega^R$.}\label{fig-domain} 
\end{figure}
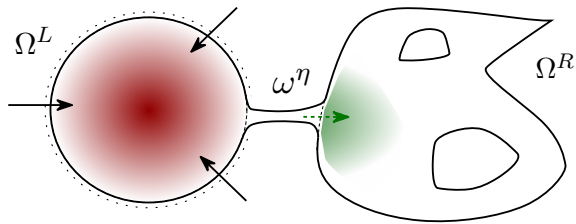

\vspace{5mm}

\needspace{5mm}
\noindent{\bf \underline{Ex}p\underline{licit controls on rectan}g\underline{ular domains}}\\[1mm]
In Section \ref{rectangle}, we study a different type of domain deformations for the specific case where $\Omega_0$ is
a rectangle. In the rectangular domain of the size $a\times b$, the spectrum of the Laplacian operator is completely known: the eigenmode $\phi_{j,k}=\sin(j\pi/a x_1)\sin(k\pi/b x_2)$ has the total energy 
$\lambda_{j,k}=\pi^2 (j^2/a^2 + k^2/b)$. Any adiabatic variation of the sizes $a$ and $b$ of the rectangle preserves both the horizontal and vertical quantum numbers $j$ and $k$, hence the position of $\lambda_{j,k}$ in the spectrum linearly ordered by the increase of the energy can be easily switched by a slow change in $a$ or $b$. For example,
for any $k_2>k_1$ and $j_2>j_1$, when $a$ grows from very small to very large values while keeping $b$ constant,
we have $\lambda_{j_2,k_1} > \lambda_{j_1,k_2}$ at the beginning of the process and $\lambda_{j_2,k_1} < \lambda_{j_1,k_2}$ at the end. 

We exploit this explicit eigenvalue crossings at rectangular deformations and obtain another strategy for the global approximate controllability, as described in the proof of Proposition \ref{control_bilinear_rec}. Starting with any state, we move it adiabatically to a vicinity of a decoupled state (a function of $x$ times a function of $y$). After that, we 
stop doing an adiabatic control and, instead, move the decoupled state to the ground state of the rectangle (see Proposition \ref{control_bilinear_rec}) by a higher-dimensional version of the technique that was developed in \cite{Beauchard} for the control of the Schr\"odinger equation in one-dimensional domains.

\vspace{5mm}

\needspace{5mm}
\noindent{\bf \underline{Adiabatic }p\underline{ermutations of ei}g\underline{enstates}}\\[1mm]
Finally, in Section \ref{section_pumping}, we discuss how simple and explicitly defined deformations permute the excited states of a particle in a rectangle. With a simple and purely adiabatic periodic motion of one side of a rectangle, we create a non-trivial permutation of the energy eigenstates. The trick is that, for a part of each perturbation cycle, we keep the rectangular shape of the domain, and for the rest of the cycle, we make the domain shape ``generic''. This means
that, when the boundary motion is slow enough, the process is adiabatic, hence it preserves
2 quantum numbers ($j$ and $k$) in the first part of the cycle, leading to eigenvalue crossings, while in the rest of the cycle no crossing occurs but the quantum numbers $j$ and $k$ are no longer defined. Altogether this means that, at the end of the cycle, the system can find itself at an energy eigenstate with a different pair of quantum numbers.
We provide heuristic arguments and numerical evidence which suggest that iteration of such permutation of eigenstates leads, typically, to an exponential Fermi acceleration. 

The same effect should be observed for adiabatic perturbations of general domains which are periodically transformed to
a dumbbell shape and back. For the part of the perturbation cycle when the domain has a dumbbell shape, the system has an additional (approximate) quantum number, which indicates whether the eigenfunction is supported mostly on the left or right part of the domain. For the part of the cycle when the domain has a generic form, this quantum number is destroyed. Similarly to the case of the rectangle, such process can lead to a non-trivial permutation of eigenstates
and to the exponential energy growth, see \cite{Dmitry}.

In general, the eigenstate permutations due to the cyclic adiabatic processes described here
(when different sets of quantum numbers are preserved on different parts of the cycle) provide an interesting 
class of number-theoretical games. The analysis of the dynamics of such permutation should be different from
the famous Collatz problem \cite{Collatzproblem}, as our permutations are automatically bijections $\NN\to\NN$, but
could be similarly difficult. In addition to dumbbell shapes and rectangles, one can use integrable domains
(ellipses and rings) and domains with discrete symmetries in order to create additional quantum numbers for a part of
the adiabatic cycle. Another possibility is to consider a pair of quantum-mechanical oscillators cyclically perturbed in such a way that they interact only for a part of the cycle. In all such processes, physical intuition suggests that the eigenstates permutations which they generate are well approximated by a positively biased geometric Brownian motion,
see Section \ref{section_pumping}. Providing a rigorous proof for such claim is a challenging number-theoretical problem,
and the results can be applicable beyond quantum mechanics, for example for the wave equation and Maxwell equation in
moving domains.

\vspace{5mm}

\needspace{5mm}
\noindent{\bf \underline{Previous works}}\\[1mm]
The origin of our article comes from the work \cite{Dmitry} where the idea was introduced that the adiabatic separation of the domain into non-symmetric parts with a consecutive reconnection of the parts can create eigenvalue crossings in an unavoidable way, leading to a non-trivial permutation of the eigenstates.  
Before being able to obtain the results of the present paper, we implemented the adiabatic separation/reconnection technique on a simple one-dimensional model where the control is provided by a moving potential  \cite{one-D}. This was the first step to understand how to completely and rigorously obtain the global approximate controllability with these techniques. Then, in \cite{SEmoving}, the first two authors introduced the framework of the Cauchy problem related to Schr\"odinger equation in moving domains. Thus, it is now possible to implement the original ideas for Equation \eqref{SE}.

Notice that the use of eigenvalue crossings to construct controls has been recently proposed also in \cite{ugo1,ugo2}.
In \cite{CT1,CT2}, the authors consider very slow motions to construct a control, these motions being ``quasi-static'' because they follow curves of steady states. Even if the PDEs considered in \cite{CT1,CT2} are not Hamiltonian, this type of control is in the same spirit as our ``quasi-adiabatic'' motions.

The control of PDEs by deformation of the domain is a difficult task and there are very few results in this direction. In \cite{BMT}, the authors study an adiabatic deformation of the domain $\Omega(t)$ in \eqref{SE} in dimension $d=1$ and for a specific case of deformation. The articles \cite{Beauchard,BC,Beauchard_Teismann,Rouchon} also consider the case $d=1$. They investigate the exact controllability problem, but only in neighborhood of some specific solutions (for comparison: our Theorem \ref{th_main} is a global result, but it does not yield an exact control since we allow a small error $\varepsilon>0$). Finally, in \cite{Moyano} the strategy of \cite{Beauchard,BC,Beauchard_Teismann,Rouchon} is followed for the higher space dimension. However, due to an assumption of radial symmetry, the techniques of \cite{Moyano} remain mostly one-dimensional.

\vspace{1cm}

\noindent{\bf Acknowledgements:} The first two authors have been supported by the project {\it ISDEEC} of the French agency {\it ANR}, project number ANR-16-CE40-0013. The third author was supported by RScF grant 19-71-10048 in the HSE - Nizhny Novgorod.

\section{The moving domains}\label{section_moving}

The Schr\"odinger equation in domains depending on time was studied in several articles, see \cite{BMT,Beauchard,Beauchard_Teismann, Knobloch_Krechetnikov, Makowski_Peplowski,Moyano, Pinder}, but often in the case of simple deformations. A general theory was developed in \cite{SEmoving}. In this section, we recall the basic tools introduced in this work and also give some new estimates. 

\subsection{The basic setting}
The first step adopted in order to deal with moving domains consists in pulling back the equation in a fixed domain $\Omega_0$. As it is classical, we use a family of $\Cc^k$-diffeomorphisms $h(t,\cdot)$ such that $h(t,\Omega_0)=\Omega(t)$ for every $t$ in\ time interval $I$ (see for example \cite{Henry,Murat_Simon} for an introduction on the subject). We need to introduce a topology associated to these deformations via diffeomorphisms. To this purpose, it is more convenient to extend $h:\Omega_0\rightarrow \Omega(t)$ into a diffeomorphism from $\Bc$ to $\Bc$ where $\Bc\subset\RR^d$ is a large closed ball containing all the domains we are interested with. 
\begin{defi}\label{defi_diffeo}
Let $\Bc\subset\RR^d$ be a large closed ball. We set 
$$\|f\|_{\Cc^k(\Bc)}=\max \big(\|f\|_{L^\infty(\Bc)},\ldots,\|D^kf\|_{L^\infty(\Bc)}\big)$$ 
to be the classical $\Cc^k-$norm. 
We denote by $\Diff^k(\Bc)$ the set of the $\Cc^k-$diffeomorphisms $h$ on $\Bc$ such that $h\equiv \id$ on $\partial\Bc$. We endow it with the $\Cc^k-$topology, considering 
$\Diff^k(\Bc)$ as a submanifold of $\Cc^k(\Bc,\Bc)$.
\end{defi}
We recall that if $h\in\Diff^k(\Bc)$, then any $g\in\Cc^k(\Bc,\Bc)$ that satisfies $g\equiv \id$ on $\partial\Bc$ and which is close enough to $h$ for the $\Cc^k-$norm, also belongs to $\Diff^k(\Bc)$. This is the reason why $\Diff^k(\Bc)$ is a submanifold of $\Cc^k(\Bc,\Bc)$ and it can be locally endowed with its topology. Now, we introduce the space of paths of diffeomorphisms in the same way.
\begin{defi}\label{defi_path}
If $I$ is a time interval, we introduce the space $\Path^k(I,\Bc)$ of \emph{$\Cc^k-$paths of diffeomorphisms} $h\in\Cc^k(I\times\Bc,\Bc)$ with $h(t)\in\Diff^k(\Bc)$ for all $t\in I$. We consider it as a submanifold of $\Cc^k(I\times\Bc,\Bc)$ and we endow it with the inherited topology.
\end{defi}

In this paper, we will always consider the following framework: $\Bc\subset\RR^d$ is a large closed ball and  $\Omega_0\subset\Bc$ is a reference domain, regular enough to be able to define a Dirichlet Laplacian operator with the classical properties that we may be interested in. Typically, we can consider a smooth reference domain or a polyhedral reference domain, but even a Lipschitz domain should be sufficient to ensure several of our statements. We consider the moving domain $\Omega(t)$ with $t\in I$ as the images $\Omega(t)=h(t,\Omega_0)$, where $h\in\Path^k(I,\Bc)$ for some $k\geq 1$. To preserve the Hamiltonian structure of the Schr\"odinger equation, it is natural to introduce a unitary version of the pull-back operator $h^*:\phi\mapsto \phi\circ h$ by considering $h^\sharp(t)$ defined by
\begin{equation}\label{pullback_sharp}
h^\sharp(t)~:~\phi\in L^2(\Omega(t),\CC)~\longmapsto~\sqrt{|J(t,\cdot)|}\,(\phi\circ h)(t) \in L^2(\Omega_0,\CC)~,
\end{equation}
where $J(t,x):=Dh(t,x)$ is the Jacobian of $h$ and $|J|$ (or $|Dh(t,x)|$) denotes the 
absolute value of its determinant. We also introduce its inverse $h_\sharp(t)$ with $t\in I$, the push-forward 
operator 
\begin{equation}\label{pushforward_sharp}
h_\sharp(t)=(h^\sharp(t))^{-1}~:~\psi\in L^2(\Omega_0,\CC) \longmapsto \big(\psi / \sqrt{|J(\cdot,t)|}\big)\circ 
h^{-1}\in L^2(\Omega(t),\CC)~.
\end{equation}
In the current work, we adopt the notation from \cite{SEmoving}. We 
denote by $x$ the points in $\Omega(t)$ and by $y$ the ones in $\Omega_0$. The notation $\la\cdot|\cdot\ra$ 
denotes the scalar product 
in $\CC^N$ with the convention 
$$\la v|w\ra=\sum_{k=1}^d\overline{v_k} \,w_k~, \ \ \ \ \ \  \ \ \ \forall v,w\in \CC^d.$$
We set $v(t)=h^\sharp(t)u(t)$ and we pull back Equation \eqref{SE} in the fixed domain $\Omega_0$. The 
straightforward computation yields an equation for which the Hamiltonian structure is not obvious at first sight. 
This structure was made more explicit in \cite{SEmoving}, by proving that the equation satisfied by $v$ is the 
following
\begin{equation}\label{SEv}
\left\{\begin{array}{ll}
i\partial_t v(t,y)=-h^\sharp\Big[\big(\div_x+iA_h\big)\circ\big(\grad_x+iA_h\big)+|A_h|^2\Big]h_\sharp 
v(t,y),~~~&(y,t)\in\Omega_0\times I,\\
v_{|\partial\Omega_0}= 0, & (y,t)\in\partial\Omega_0\times I,
\end{array}\right.
\end{equation}
where the magnetic potential $A_h$ is given by $A_h(t,x)=-\frac 12 (h_*\partial_t h)(t,x):=-\frac 12(\partial_t h(t,h^{-1}(t,x)))$. Using the equation above, we may define a flow for the Schr\"odinger equation in the moving domain $\Omega(t)$. The following result is proved in \cite{SEmoving}. 

\begin{theorem}[{\bf Theorem 1.1 of \cite{SEmoving}}] \label{th_Cauchy}
Let $\Bc\subset\RR^d$ be a large ball and let $\Omega_0\subset\Bc$ be a reference domain, either a domain of class $\Cc^2$ or a polyhedron. Let $I$ be a time interval and let $h\in \Path^2(I,\Bc)$. We set $\Omega(t)=h(t,\Omega_0)$. 

Then, Equation \eqref{SEv} generates a unitary flow $\tilde U(t,s)$ on $L^2(\Omega_0)$ and we may define weak solutions of the Schr\"odinger equation \eqref{SE} by transporting this flow via $h_\sharp$ to a unitary flow $U(t,s):L^2(\Omega(s))\rightarrow L^2(\Omega(t))$.

Assume in addition that the path of diffeomorphisms $h$ belongs to $\Path^3(I,\Bc)$. Then, for any $u_0\in H^2(\Omega(t_0))\cap H^1_0(\Omega(t_0))$ with $t_0\in I$, the flow above defines a solution $u(t)=U(t,t_0)u_0$ in $\Cc^0(I,H^2(\Omega(t))\cap H^1_0(\Omega(t)))\cap \Cc^1(I,L^2(\Omega(t)))$ solving \eqref{SE} in the $L^2-$sense.
\end{theorem}

A similar result for Neumann-type boundary conditions or for more general linear Schr\"odinger equations was also obtained in \cite{SEmoving}. The Gauge invariance, additional phase shifts or a suitable version of Moser's trick may be used to simplify \eqref{SEv} in particular situations, as presented in \cite{SEmoving}. 

Theorem \ref{th_Cauchy} shows that a relevant notion of solution of Schr\"odinger equation in $\Omega(t)$ can be obtained for $\Cc^2-$paths of domains $\Omega(t)$ and that this notion corresponds to the natural strong one in the path of domains is of class $\Cc^3$. Notice that the $\Cc^k-$smoothness does not refer to the reference domain $\Omega(0)=\Omega(t_0)$, which may have corners.

Also notice that defining $h$ outside $\Omega_0$ and equal to the identity on $\partial\Bc$ is not too much constraining. If we start from a family of diffeomorphisms $(h(t,\cdot))_{t\in I}\in\Cc^k(I\times \overline\Omega_0,\RR^d)$, then it may be impossible to embedded it in $\Path^k(\Bc)$ for some ball $\Bc$ due to topological reasons. If for instance we consider $\Omega_0\subset\RR^2$ as an annulus and $h$ reverses it inside out, then we cannot extend $h$ into a diffeomorphism of a ball. However, we may consider $\Omega(t_0)$ as a new reference domain and $\tilde h(t)=h(t)\circ h(t_0)^{-1}$ as another family of diffeomorphisms. For all the simple and smooth examples discussed in this paper, using Whitney extension theorem \cite{Whitney} and standard results of globalization of local diffeomorphisms (see \cite{Meisters_Olech} and also \cite{SEmoving}), we can then extend $\tilde h$ from $\Omega(t_0)$ into some large ball $\Bc$ in order to embed it in $\Diff^k(\Bc)$.

\subsection{Adiabatic motions}\label{sec_adiabatic}
The aim of this section is to present the adiabatic result for the Schr\"odinger equation \eqref{SEv} on moving domains under suitable assumptions on the deformation. To this purpose, we refer to \cite[Section 1 \& Section 5.1]{SEmoving}, where a very similar result is presented and proved.  

We consider a family of domains $\{\Omega(\tau)\}_{\tau\in [0,1]}$ with the framework of Theorem \ref{th_Cauchy}. We denote by $P(\tau)\in\Lc(L^2(\Omega(\tau)))$ with $\tau\in [0,1]$ a family of spectral projectors associated to the Dirichlet Laplacian operator $-\Delta$ on $\Omega(\tau)$. The classical adiabatic principle occurs when the deformation of the family of domains is sufficiently slow. We represent the slowness of the motion by a parameter $\epsilon>0$ and we consider deformations between the times $0$ and $1/\epsilon$, that is the following Schr\"odinger equation
\begin{equation}\label{SE_adiabatic}
\left\{
\begin{array}{ll}i\ddd_t u_\epsilon(t, x)=-\Delta u_\epsilon(t, x),\ \ \ \ &t \in [0,1/\epsilon]~,~~x\in 
\Omega(\epsilon t),\\
u_{\epsilon}(t) \equiv 0,&\text{ on }\partial\Omega(\epsilon t), \\
u_\epsilon(t=0)= u_0 \in L^2(\Omega(0)). &
\end{array}\right.
\end{equation}
The classical adiabatic principle say that, if we move very slowly, then the energy contained in a level of energy is almost preserved. Many versions of adiabatic results exist, see for example \cite{Nen}. In \cite{SEmoving}, we have checked the adapatation of this argument to the framework of moving domains. 
\begin{prop}\label{prop_adiabatic}
{\bf (Corollary 1.5 of \cite{SEmoving}).}
Let $N>0$. Consider a family of domains $\{\Omega(\tau)\}_{\tau\in [0,1]}$ such that, for all $\tau\in [0,1]$, the first $N$ eigenvalues $(\lambda_j(\tau))_{j=1\ldots N}$ of the Dirichlet Laplacian operator on $\Omega(\tau)$ are simple. Denote by $(\varphi_j)_{j=1\ldots N}$ and $(\psi_j)_{j=1\ldots N}$ some corresponding orthonormal eigenfunctions for $\tau=0$ and $\tau=1$ respectively. Then, the solution of \eqref{SE_adiabatic} with 
$$u_0=\sum_{j=1}^N c_j\varphi_j$$
satisfies
$$u_\varepsilon(1/\epsilon)=\sum_{j=1}^N \tilde c_j\psi_j+R~~\text{ with }~~|\tilde c_j|\xrightarrow[~~\epsilon\longrightarrow 0~~]{} |c_j|~\text{ and }~\|R\|_{L^2}\xrightarrow[~~\epsilon\longrightarrow 0~~]{}0~.$$
\end{prop}
Notice that it is also possible to extend the above result to some cases with multiple eigenvalues, see \cite[Remarks 3-4 \ p.\ 16]{Nen}. We will need this extension in the case of rectangular shape. This case is simple enough to be computed explicitly and,  in the present paper, we will restrict the extension to such case (see the statement of Proposition \ref{prop_adiabatic_square} below).

\subsection{Continuity estimates}
We need to estimate the continuity of the solutions of \eqref{SEv} with 
respect to the different deformations of the domain. By Theorem \ref{th_Cauchy}, we know that \eqref{SEv} generates a unitary semigroup and the $L^2-$norm of $v(t)$ is constant. Since we do not want to involve Poincar\'e estimates that depend on the domain, we consider
$$\|v\|_{H^1(\Omega)}=\|\grad v\|_{L^2(\Omega)} + \|v\|_{L^2(\Omega)}~,$$
even if the term $\|\grad v\|_{L^2(\Omega)}$ is sufficient to define an equivalent norm in $H^1_0(\Omega)$. Notice that if $\Omega\subset \Bc$ and $v\in H^1_0(\Omega)$, then the extension of $v$ by zero belongs to $H^1_0(\Bc)$ and has the same $H^1-$ and $L^2-$norms. We first bound the growth of the $H^1-$norm during the dynamics.
\begin{prop}\label{prop_H1bound}
Let $I$ be a time interval and $t_0\in I$. For all $R>0$, there exists $C>0$ such that the following holds. Let $h\in\Path^3(I,\Bc)$ be a family of diffeomorphisms such that 
$\|h\|_{\Cc^3(I\times\Bc,\Bc)}+\|h^{-1}\|_{\Cc^3(I\times\Bc,\Bc)}\leq R$.   Let $\Omega_0\subset\Bc$ be any reference domain. Let $v$ be the corresponding solution of \eqref{SEv} with an initial data $v(t_0)=v_0\in H^1_0(\Omega_0)$, as given by Theorem \eqref{th_Cauchy}. There holds 
$$\forall t\in 
I~,~~\|v(t)\|_{H^1(\Omega_0)} \leq C e^{C|t-t_0|}\|v(t_0)\|_{H^1(\Omega_0)}~.$$
\end{prop}
\begin{proof} 
First notice that, arguing by density, it is sufficient to prove the estimate when $v_0\in H^2(\Omega_0)\cap H^1_0(\Omega_0)$. In this case, the corresponding solution $v(t)$ is differentiable with respect to the time and \eqref{SEv} holds in the $L^2-$sense. 

A direct computation (see \cite{Henry,Murat_Simon} or \cite[Proposition 2.1]{SEmoving}) shows that 
\begin{equation}\label{H1bound_eq3}
(h^\sharp\grad_x h_\sharp)v(t,y)=\sqrt{|J(t,y)|}\big(J(t,y)^{-1}\big)^t\cdot\grad_y 
\Big(\frac{v(t,y)}{\sqrt{|J(t,y)|}}\Big).
\end{equation}
In particular, 
\begin{equation}\begin{split}\label{H1bound_eq1}
h^\sharp(\grad_x&+iA_h)h_\sharp v(t,y)= h^\sharp\grad_x h_\sharp v(t,y) -\frac i2 \partial_t h(t,y) v(t,y)\\
&=(J(t,y)^{-1})^t \grad_y v(t,y)\,-\,\frac 12 \Big[ |J(t,y)|^{-1}(J(t,y)^{-1})^t \grad_y\big(|J(t,y)|\big) + i \partial_t h(t,y)\Big]v(t,y).
\end{split}
\end{equation}
It yields the equivalence
\begin{equation}\label{H1bound_eq2}
\frac 1{C_1(R)} \|v\|_{H^1(\Omega_0)} ~\leq~ \|h^\sharp(\grad_x+iA_h)h_\sharp v\|_{L^2(\Omega_0)} + \|v\|_{L^2(\Omega_0)} ~\leq~  C_1(R) \|v\|_{H^1(\Omega_0)},
\end{equation}
where $C_1(R)$ only depends on the bounds on $h$ and its derivatives, and not on $\Omega_0$. 
The $L^2-$norm of $v(t)$ is constant in time. So, thanks to the relation \eqref{H1bound_eq1}, we have
\begin{align*}
 \partial_t \Big(\|&h^\sharp(\grad_x+iA_h)h_\sharp v(t)\|^2_{L^2}+ \|v(t)\|_{L^2}^2\Big)=
2 \Re \big\langle \partial_t\big(h^\sharp(\grad_x+iA_h)h_\sharp v(t)\big)\big|h^\sharp(\grad_x+iA_h)h_\sharp 
v(t)\big\rangle_{L^2}\\
&=2 \Re \big\langle      h^\sharp(\grad_x+iA_h)h_\sharp \ddd_t v(t) \big|h^\sharp(\grad_x+iA_h)h_\sharp 
v(t)\big\rangle_{L^2}  \\
& ~~~~ + 2 \Re \big\langle \ddd_t (J(t,y)^{-1})^t  \grad_y v(t) \big|h^\sharp(\grad_x+iA_h)h_\sharp 
v(t)\big\rangle_{L^2}\\
& ~~~~  - \Re \Big\langle    \ddd_t\Big( |J(t)|^{-1}(J(t)^{-1})^t \grad_y\big(|J(t)|\big) +i(J(t)^{-1})^t \partial_t 
h(t) \Big)
v(t)    \big|h^\sharp(\grad_x+iA_h)h_\sharp 
v(t)\Big\rangle_{L^2}.
\end{align*}
Using the above estimates and the fact that $v(t)$ is the solution of \eqref{SEv}, we obtain 
\begin{align*}
\partial_t\Big(\|h^\sharp(\grad_x+iA_h)h_\sharp v(t)\|^2_{L^2} + &\|v(t)\|_{L^2}^2\Big)\\
&\leq -  2 \Re \big\langle i h^\sharp\Big[\big(\grad_x+iA_h\big)^2+|A_h|^2\Big] h_\sharp 
v(t)\big|h^\sharp\big(\grad_x+iA_h\big)^2 h_\sharp v(t)\big\rangle_{L^2}\\
&~~~~~+ C_2(R) \|v(t)\|_{H^1(\Omega_0)}^2\\
&\leq  2 \Im \Big(\big\| h^\sharp\big(\grad_x+iA_h\big)^2 h_\sharp 
v(t)\big\|_{L^2}^2\Big)+ C_3(R) \|v(t)\|_{H^1(\Omega_0)}^2\\
&\leq C_3(R) \|v(t)\|_{H^1(\Omega_0)}^2\\
&\leq C_4(R)\Big( \|h^\sharp(\grad_x+iA_h)h_\sharp v(t)\|^2_{L^2} + \|v(t)\|_{L^2}^2\Big),
\end{align*}
where the $C_i(R)$ with $i=1,...,4$ are constants only depending on the first three 
derivatives of $h$ and $h^{-1}$, and then on the parameter $R$. It remains to apply Gr\"onwall's lemma 
and the equivalence \eqref{H1bound_eq2}.
\end{proof}

We can deduce from the previous estimate a uniform estimation of the continuity at $t=0$.  
\begin{prop}\label{prop_continuity}
For any $T>0$ and $R\geq 0$, there exists $C>0$ such that the following holds. Let $h\in\Path^3((-T,T),\Bc)$ be a family of diffeomorphisms such that $\|h\|_{\Cc^3((-T,T)\times\Bc,\Bc)}+\|h^{-1}\|_{\Cc^3((-T,T)\times\Bc,\Bc)} \leq R$.  Let $\Omega_0$ be any reference domain in $\Bc$. Then any solution $v(t)$ of \eqref{SEv} corresponding to $h$ with initial data $v(t_0)=v_0\in H^1_0(\Omega_0,\CC)$ satisfies
$$\forall t\in(-T,T)~,~~\|v(t)-v_0\|_{L^2}~\leq~ C \sqrt{|t|} \|v_0\|_{H^1}.$$
\end{prop}
\begin{proof}
As in the proof of Proposition \ref{prop_H1bound}, $h$ is smooth enough to be able to argue by density and by assuming that $v_0$ belongs to $H^2\cap H^1_0$. Using the same arguments as above, we write
\begin{align*}
\partial_t \|v(t)-v_0\|_{L^2(\Omega_0)}^2 & = 2\Re \big\la \partial_t v(t) \,\big |\, v(t)-v_0\big\ra_{L^2} \\
&= -2\Im \big\la h^\sharp\big[(\grad_x+iA_h)^2+|A_h|^2\big]h_\sharp v(t) \,\big|\, v(t)-v_0\big\ra_{L^2}\\
&=2\Im \big\la h^\sharp\big[(\grad_x+iA_h)+|A_h|\big]h_\sharp v(t) \,\big|\, h^\sharp\big[(\grad_x+iA_h)-|A_h|\big]h_\sharp(v(t)-v_0)\big\ra_{L^2}\\
&\leq C(R) \|v(t)\|_{H^1}\left( \|v(t)\|_{H^1} + \|v_0\|_{H^1}\right),
\end{align*}   
with $C(R)>0$ only depending on $R$. Finally, we obtain a uniform bound for $\partial_t 
\|v(t)-v_0\|_{L^2(\Omega_0)}^2$ by using  Proposition \ref{prop_H1bound} and the claim is ensured since 
$
\|v(t)-v_0\|_{L^2(\Omega_0)}^2\leq t\sup_{t\in (-T,T)}|\partial_t 
\|v(t)-v_0\|_{L^2(\Omega_0)}^2|.$
\end{proof}

We can also estimate the continuity of the solutions with respect to the deformations of the domain. 
\begin{prop}\label{prop_difference}
Let $I$ be a time interval and $t_0 \in I$. For any $R\geq 0$, there exists $C>0$ such that the following holds. Let $h\in\Path^3(I,\Bc)$ and $g\in\Path^3(I,\Bc)$ be two families of diffeomorphisms such that 
 $$\|h\|_{\Cc^3(I\times\Bc,\Bc)}+\|h^{-1}\|_{\Cc^3(I\times\Bc,\Bc)}\leq 
R \ \ \ \text{and} \ \ \ \|g\|_{\Cc^3(I\times\Bc,\Bc)}+\|g^{-1}\|_{\Cc^3(I\times\Bc,\Bc)}\leq R~.$$
Let $\Omega_0$ be any reference domain in $\Bc$. Let $v(t)$ be the solutions of \eqref{SEv} corresponding to $h$ with initial data $v(t_0)=v_0\in H^1_0(\Omega_0,\CC)$ and $w$ be another solution 
corresponding to $g$ with initial data $w(t_0)=w_0\in H^1_0(\Omega_0,\CC)$, as given by Theorem \eqref{th_Cauchy}. Then,
$$\|v(t)-w(t)\|_{L^2}^2~\leq \|v_0-w_0 \|_{L^2}^2+ 
C \left(e^{C |t-t_0|}-1\right)\| h-g\|_{\Cc^2(I\times\Bc,\Bc)} \|v_0\|_{H^1}\|w_0\|_{H^1}.$$
\end{prop}
\begin{proof}
We use the same arguments leading to the previous propositions. First, we notice that $h$ and $g$ are smooth 
enough to be able to argue by density and by assuming that $v_0$ and $w_0$ belong to $H^2\cap H^1_0(\Omega_0,\CC)$. 
Then, since the flow is unitary,
\begin{align}
\ddd_t\|v(t)-&w(t)\|^2_{L^2}
=-2Re\Big(\la\ddd_t v(t)|w(t)\ra_{L^2}+\la 
v(t)|\ddd_tw(t)\ra_{L^2}\Big) \nonumber\\
=&-2Im\Big\la h^\sharp\Big[\big(\grad_x+iA_h\big)^2+|A_h|^2\Big]h_\sharp             
 v(t)\Big|w(t)\Big\ra_{L^2}\nonumber\\
 &+2Im\Big\la         
 v(t)\Big| g^\sharp\Big[\big(\grad_x+iA_g\big)^2+|A_g|^2\Big]g_\sharp 
 w(t)\Big\ra_{L^2}\nonumber\\
  =&-2Im\Big\la h^\sharp\big(\grad_x+iA_h\big)h_\sharp            
 v(t)\Big|h^\sharp \big(\grad_x+iA_h\big) h_\sharp w(t)\Big\ra_{L^2}\nonumber\\
 &+2Im \Big\la g^\sharp\big(\grad_x+iA_g\big)g_\sharp             
 v(t)\Big|g^\sharp\big(\grad_x+iA_g\big) g_\sharp w(t)\Big\ra_{L^2}\nonumber\\ 
  &-2Im \Big\la h^\sharp|A_h|^2 h_\sharp      
 v(t) \Big|  w(t)\Big\ra_{L^2} +2Im \Big\la v(t) \Big| g^\sharp|A_g|^2 g_\sharp w(t)\Big\ra_{L^2} \nonumber\\
 =&-2Im\Big\la h^\sharp\big(\grad_x+iA_h\big)h_\sharp            
 v(t)\Big|\Big(h^\sharp\big(\grad_x+iA_h\big)h_\sharp- g^\sharp \big(\grad_x+iA_g\big)g_\sharp  \Big)  w(t)\Big\ra_{L^2} \nonumber\\
&-2Im\Big\la\Big(h^\sharp \big(\grad_x+iA_h\big)h_\sharp - g^\sharp \big(\grad_x+iA_g\big)g_\sharp \Big)      
 v(t)\Big| g^\sharp \big(\grad_x+iA_g\big)g_\sharp  w(t)\Big\ra_{L^2} \label{long_eq_continuity1}\\
 &+2Im \big\la \big(g^\sharp|A_g|^2 g_\sharp -  h^\sharp|A_h|^2 h_\sharp\big)       
 v(t) \big|  w(t)\big\ra_{L^2}. \nonumber
\end{align}
We study the objects appearing in \eqref{long_eq_continuity1}. First, since $h^\sharp |A_h|^2h_\sharp=\frac 14 (\partial_t h)^2$ and $g^\sharp |A_g|^2h_\sharp=\frac 14 (\partial_t g)^2$, we easily bound the last term by 
$$  \Big|\big\la \big(g^\sharp|A_g|^2 g_\sharp -  h^\sharp|A_h|^2 h_\sharp\big)       
 v(t) \big|  w(t)\big\ra_{L^2} \Big|~\leq~C_1(R)\| h-g\|_{\Cc^1(I\times\Bc,\Bc)}\|v_0\|_{L^2}\|w_0\|_{L^2}~.$$
Secondly, as shown by \eqref{H1bound_eq2}, 
$$\big\|h^\sharp\big(\grad_x+iA_h\big)h_\sharp            
 v(t)\big\|_{L^2} \leq C_2(R) \|v(t)\|_{H^1} ~~\text{ and }~~\big\|g^\sharp\big(\grad_x+iA_g\big)g_\sharp w(t)\big\|_{L^2} \leq C_2(R) \|w(t)\|_{H^1}~.$$
Then, we write
\begin{align*}
\big\|\big(h^\sharp\big(\grad_x+iA_h\big)h_\sharp- g^\sharp 
\big(\grad_x+iA_g\big)g_\sharp&  \big)  w(t)\big\|_{L^2}
= \big\|\big(h^\sharp \grad_xh_\sharp -g^\sharp 
\grad_xg_\sharp\big)w(t)+\frac{i}{2} \big(\ddd_t h-\ddd_t g\big) 
 w(t)\big\|_{L^2}\\
&\leq \big\|\big(h^\sharp \grad_xh_\sharp -g^\sharp 
\grad_xg_\sharp\big)w(t)\big\|_{L^2} + \|h-g\|_{\Cc^1(I\times\Bc,\Bc)}\|w_0\|_{L^2}.
\end{align*}
It remains to estimate the terms of the type $\big\|\big(h^\sharp \grad_xh_\sharp -g^\sharp \grad_xg_\sharp\big)w(t)\big\|_{L^2}$. 
By using \eqref{H1bound_eq3}, we obtain
\begin{align*}
\big\|\big(h^\sharp \grad_xh_\sharp -g^\sharp \grad_xg_\sharp&\big)w(t)\big\|_{L^2}\\
&\leq 
\Big\|\sqrt{|Dh|}h^*\grad_xh_*\Big(\sqrt{|Dh|}^{-1}\Big)-\sqrt{|Dg|}g^*\grad_xg_*\Big(\sqrt{|Dg| }^{-1} \Big)\Big\|_{L^\infty}\|w(t)\|_{L^2}\\
&~~~~+\big\| \big(\big(Dh^{-1}\big)^t -\big(Dg^{-1}\big)^t\big)\big\|_{L^\infty}\|\grad_y w(t)\big\|_{L^2}\\
&\leq C_3(R) \|h-g\|_{\Cc^2(I\times\Bc,\Bc)} \|w(t)\|_{H^1}.
\end{align*}
Again, we underline that the constants $C_i(R)$ with $i=1,2,3$ do not depend on $\Omega_0$ or the initial data. By 
using the Cauchy-Schwarz inequality in \eqref{long_eq_continuity1} and the above estimates, we obtain a constant 
$C_4(R)>0$, only depending on $R$, such that
$$\ddd_t\|v(t)-w(t)\|^2_{L^2}\leq C_4(R) \|h-g\|_{\Cc^2(I\times\Bc,\Bc)} \|v(t)\|_{H^1}\|w(t)\|_{H^1}.$$
Finally, we apply Proposition \ref{prop_H1bound} to get that there exist $C_5(R),C_6(R)>0$, only depending on $R$, 
such that
$$\ddd_t\|v(t)-w(t)\|^2_{L^2}\leq C_5(R) e^{C_6(R)|t-t_0|} 
\|h-g\|_{\Cc^2(I\times\Bc,\Bc)} \|v_0\|_{H^1}\|w_0\|_{H^1}$$
and it remains to integrate this last estimate in order to ensure the claim. 
\end{proof}

\section{Moving domains and the spectrum of the Dirichlet Laplacian}\label{section3}

\subsection{Generic properties}\label{section_simplicity}
Micheletti in \cite{Micheletti} was one of the first to show that the spectrum of the Laplacian operator is simple, generically with respect to the geometry of the domain. See also the work of Uhlenbeck \cite{Uhlenbeck} and the very complete book of Henry \cite{Henry}. They consider fixed domains $\Omega=h(\Omega_0)$, with $h$ a $\Cc^k-$diffeomorphism in $\Diff^k(\Bc)$ as in Definition \ref{defi_diffeo}. By a \emph{generic set of domains}, we mean a generic subset of the Banach manifold $\Diff^k(\Bc)$. We recall that a subset of a Banach manifold $X$ is called \emph{generic} if it contains a countable intersection of dense open subsets of $X$.
\begin{theorem}[{\bf \cite{Micheletti},\cite{Uhlenbeck}, Chapter 6 of \cite{Henry}}]\label{th_simplicity}
Let $\Bc$ be a closed ball and let $\Omega_0$ be an open $\Cc^2-$domain, or a polyhedron, with $\overline\Omega_0\subset \mathring\Bc$. For any $k\geq 2$, there is a generic set of diffeomorphisms $h\in\Diff^k(\Bc)$ such that the Laplacian operator $\Delta$ in $h(\Omega_0)$ has only simple eigenvalues.
\end{theorem}

In this paper, we need to follow paths of domains $\Omega(t)$ without meeting multiple eigenvalues. The genericity result above is not sufficient: we need to know that the domains with multiple eigenvalues belong to a set of codimension at least $2$. To study the codimension of this set, \cite{CdV} introduces the {\it strong Arnold hypothesis}. As noticed in \cite{Teytel}, when we only want to obtain a codimension larger than $2$, we may consider a weaker hypothesis: the (SAH2) presented below. In \cite{Teytel}, Teytel shows that for any couple diffeomorphic domains $\Omega(0)$ and $\Omega(1)$, we can find an analytic path $(\Omega(\tau))_{\tau\in[0,1]}$ linking them, such that, for all $\tau\in (0,1)$, the Laplacian operator on $\Omega(\tau)$ has a simple spectrum. In fact, the proof yields a stronger result. Firstly, this path can be made as close as wanted to a target path. Secondly, it is possible to consider a subfamily of possible domains as soon as this family satisfies the hypothesis (SAH2) explicitly stated in \cite{Teytel}. Lastly, even if this is not useful for us, notice that \cite{Teytel} states abstract results with many other applications than the paths of domains. 
\begin{theorem}[{\bf Theorem 6.4 of \cite{Teytel}}]\label{th_Teytel}
Let $\Bc$ be a closed ball and let $\Omega_0$ be a connected open $\Cc^2-$domain, or a polyhedron, with $\overline\Omega_0\subset \mathring\Bc$.
Let $k\geq 2$ and let $h\in\Path^k([0,1],\Bc)$ representing a path of domains $\Omega(\tau)=h(\tau,\Omega_0)$. Then, for all $\varepsilon>0$, there exists a close path $g\in\Path^k([0,1],\Bc)$ such that the spectrum of the Dirichlet Laplacian operator $-\Delta$ in $\tilde\Omega(\tau)=g(\tau,\Omega_0)$ is simple for all $\tau\in(0,1)$ and 
$$g(0)=h(0),\ \ \ \ \  g(1)=h(1),\ \ \ \ \   \|g-h\|_{\Cc^k([0,1]\times\Bc,\Bc)}< \varepsilon.$$
\end{theorem}
There are few differences with the original statement of Teytel, that are discussed in the proof below. 
We also would like to restrict the possible domains to stay, for example, in the class of polygonal domains. To this end, we restrict the possible diffeomorphisms to a submanifold $\Hc$ of $\Diff^k(\Bc)$. At each $h\in\Hc$, the tangent space $T_h\Hc$ is a subspace of $\Cc^k(\Bc,\RR^d)$. In this framework, the hypothesis (SAH2) is as follows (see Sections 1 and 6 of \cite{Teytel}).
\begin{itemize}
\item[(SAH2)] Let $h\in\Hc \subset \Diff^k(\Bc)$ and $N\in\NN$. We say that (SAH2) is satisfied at $h$ along the submanifold $\Hc$ for the $N$ first eigenvalues when the following property is verified. If the Dirichlet Laplacian operator $-\Delta$ in 
$\Omega=h(\Omega_0)$ has a multiple eigenvalue $\lambda$ among its first $N$ eigenvalues, then there are two 
orthogonal eigenfunctions $\varphi_1$ and $\varphi_2$ corresponding to the eigenvalue $\lambda$ such that the three linear functionals
$$g\in T_h\Hc~\longmapsto~\int_{\partial\Omega} \Dnu{\varphi_i}\Dnu{\varphi_j} \big\la (h_*g)(\sigma)|\nu(\sigma)\big\ra \dd\sigma~~~~\text{with }(i,j)=(1,1),~(2,2)\text{ or }(1,2)$$
are linearly independent, where $\nu(\sigma)$ denotes the normal vector to $\partial \Omega$ at $\sigma$.
\end{itemize}
We can state a modified version of the result of Teytel.
\begin{theorem}\label{th_Teytel_bis}
Let $k\geq 2$, $N\in\NN$ and let $h\in\Path^k([0,1],\Bc)$ representing a path of domains $\Omega(\tau)=h(\tau,\Omega_0)$. Assume that, for all $\tau\in [0,1]$, $h(\tau)$ belongs to the subclass $\Hc\subset\Diff^k(\Bc)$ and that Hypothesis (SAH2) holds at $h(\tau)$ along $\Hc$ for the $N$ first eigenvalues. Then, for all $\varepsilon>0$, there exists a path $g\in\Path^k([0,1],\Bc)$ such that, for all $\tau\in(0,1)$, there holds $g(\tau)\subset \Hc$, the $N$ first eigenvalues of the Dirichlet Laplacian operator $-\Delta$ in $\tilde\Omega(\tau)=g(\tau,\Omega_0)$ are simple and
$$g(0)=h(0),\ \ \ \ \ g(1)=h(1),\ \ \ \ \ \|g-h\|_{\Cc^k([0,1]\times\Bc,\Bc)}< \varepsilon.$$
\end{theorem}
\begin{proof}
All the arguments for proving both previous results are contained in \cite{Teytel}, but since the statements are different from the one of Teytel, we emphasize here some key points. First, Theorem 6.4 of \cite{Teytel} considers two domains $\Omega(0)$ and $\Omega(1)$ homotopic to the ball. This hypothesis is assumed to ensure that there exists at least a path connecting both domains. In our case, the existence of such a path is part of the hypotheses so we can be more general concerning the topology of these domains (this was already noticed in the erratum of \cite{Privat_Sigalotti}). 
Second, Theorem 6.4 of \cite{Teytel} does not consider a subclass $\Hc$ of domains and directly proves that (SAH2) is satisfied with respecto to the whole class of diffeomorphic domains. However, Assumption (SAH2) and the main result Theorem B of \cite{Teytel} are stated in a very general way including the possibility of few degrees of freedom. In Section 6 of \cite{Teytel}, Teytel considers the case of domain perturbations and computes (SAH2) as stated above. He also checks that it is satisfied when $\Hc$ is the whole class of deformations of the domain as in Theorem \ref{th_Teytel}. Notice that (SAH2) is obviously not satisfied for $\varphi_1$ and $\varphi_2$ supported in different part of the domain and this is why the connectedness requested in Theorem \ref{th_Teytel} above is mandatory.

We would also like to underline that the arguments of \cite{Teytel} are local ones and that is why we can state Theorems \ref{th_Teytel} and \ref{th_Teytel_bis} in a perturbative form. If (SAH2) is satisfied at some point $h\in\Hc$, then it yields local informations in a neighborhood of $h$ as it is classical when applying the transversality theorems, see for example 3.2 of \cite{Teytel}. Since in Theorem \ref{th_Teytel_bis} we aim at staying close to a compact path $\tau\in[0,1] \mapsto h(\tau)$ and since we only consider a finite number of eigenvalues, it is sufficient to check (SAH2) at each point $h(\tau)$ and to apply the arguments in a tubular neighborhood of the original path. 

It remains to emphasize that the path constructed in the proof of Theorem B of \cite{Teytel} is actually constructed as the perturbation of a first path. The original path of Teytel is piecewise linear and its difficult to control the derivatives of the constructed perturbation. To be complete, let us show how to adapt the local argument of Teytel to our purpose. Let $h(t)$ be a given path. We perturb it locally close to a time $t_0$. There exist a small $\tau>0$, a tubular neighborhood $\Tc\in\Diff^k(\Bc)$ of 
$\{h(t),\,t_0-\tau<t<t_0+\tau\}$ and a smooth function $\gamma\in\Cc^\infty([t_0-\tau,t_0+\tau])$, with $\gamma$ and all its derivative vanishing at $t_0\pm\tau$, such that the following holds. There is a hyperspace $\Dc\in\Diff^k(\Bc)$, complementary to span$(\partial_th(t_0))$, such that any function $g$ in $\Tc$ is uniquely represented by coordinates $(t,\delta)\in (t_0-\tau,t_0+\tau)\times\Dc$ via $g(y)=h(t,y)+\gamma(t)\delta(y)$. The function $g\in\Tc\mapsto \delta\in\Dc$ is a "nonlinear projection", that is a Fredholm map of index $1$. Due to (SAH2), the set of diffeomorphisms in $\Diff^k(\Bc)$ such that the Dirichlet Laplacian operator has multiple eigenvalues is of codimension at least 2. Thus, its projection by $g\in\Tc\mapsto \delta\in\Dc$ has a meager image. Thus, there exists $\delta$ as small as wanted such that, for all $t\in(t_0-\tau,t_0+\tau)$, the path $t\mapsto h(t)+\gamma(t)\delta$ avoids the diffeomorphisms providing multiple eigenvalues. We can repeat this local perturbation a finite number of times. It is sufficient to cover the whole time interval $[0,1]$ because the length $\tau$ is uniform with respect to the second time derivative of $h$, which is bounded by assumption.   
\end{proof}

\vspace{3mm}

We will also need domains without rational resonances in the spectrum. Actually, it is a generic property, as it can be proved by the techniques of Henry in \cite{Henry}. It is stated as a consequence of a much general result in \cite{Privat_Sigalotti}.
\begin{theorem}[\bf Corollary 8 of \cite{Privat_Sigalotti}]\label{th_irrational}
Let $\Bc$ be a closed ball and $\Omega_0$ a Lipschitz domain with $\overline\Omega_0\subset \mathring\Bc$. For any $k\geq 2$, there is a generic set of diffeomorphisms $h\in\Diff^k(\Bc)$ such that the Laplacian operator $-\Delta$ in $h(\Omega_0)$ has only simple eigenvalues $(\lambda_j)\subset\RR$ that are rationally independent. 
\end{theorem}

\subsection{Singular convergence}
In this section, we consider the case of singular convergence of domains. For all $\eta\in[0,1]$, let $(\Omega^\eta)\subset\RR^d$ be bounded domains with Lipschitz boundaries. Let $(\lambda^\eta_j)_{j\in\NN^*}$ be the eigenvalues of the corresponding Dirichlet Laplacian operators in $\Omega^\eta$, ordered and counted by multiplicity. For $\lambda\not\in\{\lambda^\eta_j\}$, we denote by $R^\eta(\lambda)\in\Lc(L^\infty(\RR^d))$ the corresponding resolvent operator defined as follows. Any function $f\in L^\infty(\RR^d)$ is first truncated inside $\Omega^\eta$, then we apply the classical resolvent $(\lambda - \Delta)^{-1}$ to obtain a function in $L^\infty(\Omega^\eta)$, which is extended by zero to go back to $L^\infty(\RR)^d$ afterwards. This extension enables to compare resolvent in a space independent of $\eta$ and it is sufficient to obtain the convergence of the spectrum.

Arendt and Daners show in \cite{Arendt-Daners} and \cite{Daners} the following result.
\begin{theorem}[\bf Theorem 5.10 and Section 7 of \cite{Arendt-Daners} and Theorem 7.5 of \cite{Daners}]\label{th_cv_spectrum}
Assume that for all compact $K\subset \Omega^0$, there is $\eta_0>0$ such that for all $\eta\in (0,\eta_0)$, $K\subset \Omega^\eta$. Assume the same for the exteriors: for all compact $K\subset \RR^d\setminus\Omega^0$, there is $\eta_0>0$ such that for all $\eta\in (0,\eta_0)$, $K\subset \RR^d\setminus\Omega^\eta$. 

Then, the spectrum of the Dirichlet Laplacian operators converges when $\eta$ goes to zero in the following sense:
\begin{enum2}
\item for all $j\geq 1$, $\lambda^\eta_j\longrightarrow \lambda^0_j$ when $\eta\longrightarrow 0$.
\item For all $\lambda\not\in\{\lambda^0_j\}$, $R^\eta(\lambda)$ is well defined for $\eta$ small enough and $R^\eta(\lambda)$ converges to $R^0(\lambda)$ in $\Lc(L^\infty(\RR^d))$. In particular, the spectral projectors of the Dirichlet Laplacian operators converge when $\eta$ goes to zero. If $\lambda^0_j$ is a simple eigenvalue with an eigenfunction $\varphi^0_j$, then there exist eigenfunctions $\varphi^\eta_j$ such that $\varphi^\eta_j\longrightarrow \varphi^0_j$ in $H^1_0(\RR^d)$ when $\eta\longrightarrow 0$.
\end{enum2}
\end{theorem}
For further details, we refer to Theorem 5.10 and Section 7 of \cite{Arendt-Daners}, and Theorem 7.5 of \cite{Daners}  (see also \cite{Arrieta}). We notice that the domains considered in this paper are ``strongly regular'' in the sense of \cite{Arendt-Daners} because they have Lipschitz boundaries. We intend to use Theorem \ref{th_cv_spectrum} in the case of  dumbbell shaped domains, which is a very classical example.

\section{Proof of Theorem \ref{th_main}}\label{section_proof}

\subsection{Preliminaries}\label{section_prelim}
A first important remark is that, since the flow is unitary, the smallness of the errors in $L^2(\Omega(t))$ is preserved by the flow for all $t'>t$. Thus, we may simply count the accumulated errors at each time that an approximation is made, without wondering what happens to the neglected term in the future.  

Let $\Omega_0\subset\RR^d$ be the reference domain of Theorem \ref{th_main} and $\Bc$ a large ball containing it. 
Let $u_0$ and $u_1$ respectively be the starting and aimed states in $L^2(\Omega_0)$. Let $\varepsilon>0$ be the accepted error. Using 
the generic simplicity stated in Theorem \ref{th_simplicity}, we can find a homotopic domain $\Omega_0'$ in which 
the associated Dirichlet Laplacian operator has a simple spectrum with a Hilbert basis of eigenfunctions 
$(\varphi_j)_{j\geq 1}$. Let $h\in \Path^k([0,1],\Bc)$ be such that 
$h(0,\Omega_0)=\Omega_0$ and $h(1,\Omega_0)=\Omega_0'$. 

We notice that if $u(t)$ is solution of 
the Schr\"odinger equation \eqref{SE} in $\Omega(t)=h(t,\Omega_0)$ with $u^i=u(t=0)$ and $u^f=u(t=1)$, then $v(t)=\overline u (1-t)$ is solution of the same equation \eqref{SE} in $\Omega(1-t)$ with initial data $v(t=0)=\overline u^f$ and endpoint $v(t=1)=\overline u^i$. This time reversibility of the Schr\"odinger equation allows to define the solutions of \eqref{SE} when we ``reverse'' the deformation of the domain. Notice that the roles of the initial and final states are swapped, up to conjugation.

Let $u_0(t)$ be the solution of \eqref{SE} in $\Omega(t)=h(t,\Omega_0)$ with initial data $u(t=0)=u_0$ and set $u'_0=u_0(t=1)$. Let $v(t)$ be another solution of the Schr\"odinger equation \eqref{SE} in $\Omega(t)$ with initial data $v(t=0)=\overline {u_1}$ and set $v'=v(t=1)$. Thanks to the time reversibility of the equation, $u_1(t):=\overline v(1-t)$ is the solution of \eqref{SE} in $\Omega(1-t)$ steering $u'_1:=\overline v'$ in $u_1$. 

There exist $N\in\NN$ and some coefficients $(c_j)_{j=1\ldots N}$ and $(d_j)_{j=1\ldots N}$ such that 
$$\sum_{j=1}^N |c_j|^2=\sum_{j=1}^N |d_j|^2~~,~~~~\Big\|u'_0 - \sum_{j=1}^N c_j\varphi_j\Big\|_{L^2(\Omega'_0)}\leq \frac \varepsilon 4 ~~~~\text{ and }~~~~\Big\|u'_1 - \sum_{j=1}^N d_j\varphi_j\Big\|_{L^2(\Omega'_0)}\leq \frac \varepsilon 4,$$
where we use $\|u'_0\|=\|u'_1\|$ because the flow given by Theorem \ref{th_Cauchy} is unitary and $\|u_0\|=\|u_1\|$ by assumption. Assume that the following claim holds.

\medskip
\begin{claim}\label{claim_main}
Let $N\in\NN$ and $\varepsilon>0$ be given and let $A=\big(\sum_{j=1}^N |b_j|^2\big)^{1/2}$. Then there exist $T>0$ and a path $h'(t)\in\Path^k([0,T],\Bc)$ such that the following holds. The motion is a loop in the sense 
that $h'(0)=h'(T)=\id$. Moreover, if $u'(t)$ is the solution of the Schr\"odinger equation \eqref{SE} in 
$\Omega'(t)=h'(t,\Omega'_0)$ with initial data $u'(0)=A \varphi_1$, then 
$$\Big\|u'(T) - \sum_{j=1}^N b_j\varphi_j\Big\|_{L^2(\Omega'_0)}\leq \frac \varepsilon 4~.$$
\end{claim}

Denote by $h'_0$ and $h'_1$ the deformations driving $A\varphi_1$ to $\sum_{j=1}^N \overline{c_j}\varphi_j$ and 
$\sum_{j=1}^N d_j\varphi_j$ respectively, up to an error $\varepsilon/4$. We concatenate $h(t)$, 
$h'_0(T-t)\circ h(1)$, $h'_1(t)\circ h(1)$ and $h(1-t)$ to obtain, thanks to the time reversibility, a motion of the domains steering approximately $u_0$ in $u_1$. Indeed, this deformation drives $u_0$ successively to $u'_0$ which is close to 
$\sum_{j=1}^N c_j\varphi_j$ (up to an error $\varepsilon/4$), then to $\overline{A\varphi}=A\varphi_1$ (up to an error 
$\varepsilon/2$), then to $\sum_{j=1}^N d_j\varphi_j$ (up to an error $3\varepsilon/4$) which is close to $u'_1$ 
(up to an error $\varepsilon$), and finally to $u_1$ (up to an error $\varepsilon$).

To summarize, these preliminaries reduce the whole proof of Theorem \ref{th_main} to the above claim (which is a particular case of Theorem \ref{th_main}). Proving Claim \ref{claim_main} is the purpose of the remaining part of this section.

\subsection{Sketch of the global strategy}\label{section_strategy}
One of the main idea of our strategy was introduced in \cite{Dmitry}. Assume that the spectrum of our operator splits into two separated parts that belong to two separated subspaces of the phase space: domain with two disconnected parts, separation between even and odd eigenfunctions, between states not depending on $x_1$ and states not depending on $x_2$\ldots{} Then, when we make adiabatic motions, the distribution of energy follows the curves of eigenvalues, even when eigenvalues of one part cross eigenvalues of the other part, due to their independence. Following this idea, we can shuffle the energy carried by the eigenfunctions and drive a state $\sum_{j=1}^N b_j\varphi_j$ to another state $\sum_{j=1}^N b_{\sigma(j)}\varphi_j$ where $\sigma$ is a permutation of the indices. 

In \cite{one-D}, we studied the simple situation of the Schr\"odinger equation on $[0,1]$ with a potential $V(x)$. When $V(x)$ is a very high and localized wall, the segment is almost split into two parts, but not perfectly. We showed that the idea of \cite{Dmitry} can still be used, up to carefully avoiding the tunneling effect when both part of the segment have a resonance. Moreover, a new observation was made in comparison with \cite{Dmitry}: we showed that one, in fact, can use this tunneling effect to distribute the energy between two eigenmodes when they (almost) cross. 

We use here the same strategy. We intend to create a situation where the spectrum of the Dirichlet Laplacian operator on the domains $\Omega(t)$ behaves in Figure \ref{fig-spectrum}. We separate the spectrum between two parts, the left and right ones. The right eigenvalues $\lambda^R_j$ are almost constant. The lowest left eigenvalue $\lambda_1^L$ corresponds to the ground state $A\varphi_1$ at $t=0$ which is our starting state. Then, we deform the domain to increase the eigenvalue $\lambda_0^L$, see Figure \ref{fig-spectrum}. If we do this in a adiabatic way, the distribution of energy is not modified (see Lemma \ref{lemma_adiab} below). But when $\lambda_1^L$ meets $\lambda^R_1$, a tunneling effect appears and, by tuning the speed with which the domain boundary moves, we are able to distribute the energy between the levels $\lambda_1^L$ and $\lambda^R_1$, following the method of \cite{one-D}. This is the key argument of the proof of Theorem \ref{th_main} which is ensured in Lemma \ref{lemma_step} below. We continue to increase the left part of the spectrum until $\lambda^L_1$ has crossed all the desired levels $\lambda^R_j$ to distribute the energy as in our aimed state $$u_1=\sum_{j=1}^N a_j\varphi_j.$$

\begin{figure}[ht]
\begin{center}
\resizebox{0.6\textwidth}{!}{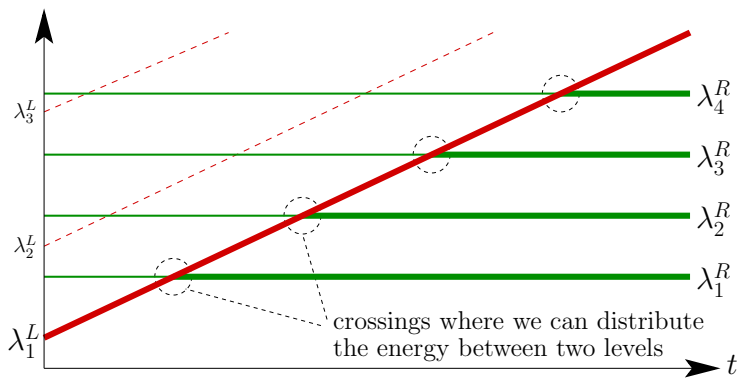}
\end{center}
\caption{\it The ``ideal'' behavior of the spectrum during our control process. The eigenvalues are split between two groups $\lambda^L_j$ and $\lambda^R_j$. We deform the domain to ensure that the eigenvalue $\lambda^L_1$ cross a suitable number of eigenvalues $\lambda_j^R$ of the other group. At each crossing point, an accurate use of the tunneling effect enables to distribute the desired part of the energy carried by $\lambda^L_1$ to $\lambda^R_j$. The actual deformation of the domain used in our proof mimics this ``ideal'' situation, except that the crossing of the eigenvalues will be broken into an ``almost crossing'', which still yields a tunneling effect.} \label{fig-spectrum} 
\end{figure}

Indeed, in our proof, we will not be able to reproduce the ``ideal'' situation of Figure \ref{fig-spectrum}. To perform the strategy above, we approximate Figure \ref{fig-spectrum} via the classical dumbbell shaped domains. Other frameworks are possible (see Section \ref{rectangle} for example), but the dumbbell shaped domains yield a simple and general proof. By assumption, our initial reference domain $\Omega_0$ provides a simple spectrum. This domain, denoted by $\Omega^R$, corresponds to the right part of the spectrum $(\lambda^R_j)_{j\in\NN^*}$. From a smooth part of the boundary, we grow an attached ball $\Omega^L$ linked with a thin channel $\omega^\eta$, see Figure \ref{fig-domain}. The spectrum of the Dirichlet Laplacian $-\Delta$ on this ball is denoted $(\lambda^L_j)_{j\in\NN^*}$. The interest of this dumbbell shaped domain is that, if the channel is very thin, then the spectrum of the whole domain can be approximated by $(\lambda^L_j)_{j\in\NN^*}\cup(\lambda^R_j)_{j\in\NN^*}$. During this deformation, if we move adiabatically as in Proposition \ref{prop_adiabatic} and if $\Omega^L$ is sufficiently large, then the ground state of the initial domain is  mainly supported by the ball $\Omega^L$, because $\lambda_1^L$ is the lowest eigenvalue.

Afterwards, we reduce the size of the ball $\Omega^L$ in order to obtain the behavior of the spectrum as described above (Figure \ref{fig-spectrum}). We use the tunneling effect when the ground energy level $\lambda^L_1$ 
of the left ball is equal to an eigenvalue $\lambda^R_k$ of the right domain. In these resonant moments, it is then possible to distribute a part of the energy contained in the left ball to $k-$th energy level of the right part. Notice that, for the actual eigenvalues of the Laplacian operator in the dumbbell domain with the thin channel, the crossing of the ``ideal'' eigenvalues $\lambda^L_1$ and $\lambda^R_k$ becomes an ``almost crossing'' because the exact crossing is not a generic situation (see Section \ref{section_simplicity}). However, they will be close enough to observe a tunneling effect.

\subsection{The basic step: distribution of energy during the (almost) crossing of eigenvalues}
The purpose of this section is to rigorously obtain the key step of the strategy described in Section \ref{section_strategy}: the distribution of the energy via the tunneling effect.
We make precise in the following lemma the situation of Figure \ref{fig-domain} and analyze the evolution of the state during the deformation. In what follows, $\Omega_0$ denotes the starting domain of Theorem \ref{th_main} and $\Bc$ is a large ball containing it.
\begin{lemma}\label{lemma_step}
Let $c_j\in\RR$, $j=1,\ldots,K$ be given, let $\alpha\in [0,1]$ and let $\delta>0$. There exist a time interval $I=[t_0,t_1]$ and a family of diffeomorphisms $h(t)\in\Path^3(I,\Bc)$ such that the following holds. 

The spectrum of the Dirichlet Laplacian operator is simple on $\Omega(t_0)=h(t_0,\Omega_0)$ and $\Omega(t_1)=h(t_1,\Omega_0)$ respectively. Denote by $(\varphi_j)_{j\in\NN^*}$ and $(\psi_j)_{j\in\NN^*}$ two respective Hilbert bases of eigenfunctions in these domains. Let $u(t)$ be the solution of the Schr\"odinger equation on the moving domain $\Omega(t)=h(t,\Omega_0)$ with initial data $u(t_0)=\sum_{j=1}^K c_j\varphi_j$. Then at time $t_1$, there exist phase shifts $\theta_j\in\RR$ such that
$$\left\|u(t_1)-\left(\alpha e^{i\theta_K}c_K\psi_K+\sqrt{1-\alpha^2}e^{i\theta_{K+1}}c_K\psi_{K+1} + \sum_{j=1}^{K-1}e^{i\theta_j}c_j \psi_j \right)\right\|_{L^2(\Omega(t_1))}\leq \delta~.$$
\end{lemma}
\begin{proof}
We set $\delta'=\delta/ \big(4(K+1)(\max |c_j|)\big)$. Consider first a domain $\Omega_0^0=\Omega^R\cup\Omega^L$ split into two parts. The right side is a smooth domain $\Omega^R$ diffeomorphic to the starting domain of Theorem \ref{th_main} (possibly with holes and corners). Up to use Theorem \ref{th_simplicity}, we can assume that the spectrum of the Dirichlet Laplacian operator on $\Omega_R$ is simple and we denote by $(\lambda^R_j)_{j\in\NN^*}$ its ordered eigenvalues and by $(\varphi^R_j)_{j\in\NN^*}$ a corresponding Hilbert basis of eigenfunctions. The left side $\Omega^L$ is a simple domain, typically a ball. Its size is chosen such that the first eigenvalue $\lambda_1^L$ of the Dirichlet Laplacian operator is precisely equal to $\lambda^R_K$, with a corresponding eigenfunction $\varphi_1^L$. Choose a large ball $\Bc$ containing both parts of the domain. We consider a family of diffeomorphisms $h(t)\in\Path^3([-\tau,\tau],\Bc)$ for some small $\tau>0$ such that: $h(t=0)=\id$, the right part is never modified, {\it i.e.} $h(t)_{|\Omega^R}=\id$ for all $t\in [-\tau,\tau]$, and the left part is simply homothetically transformed by $h(t)_{|\Omega^L}=(1-t)\id$ for all $t\in [-\tau,\tau]$. We set $\Omega^0(t)=h(t,\Omega^0_0)$. For small $\tau>0$, this construction yields the following properties for the spectrum of the Dirichlet Laplacian operator in $\Omega^0(t)$: 
\begin{enum2}
\item For all $t$, its first $K-1$ eigenvalues are $\lambda^R_j$ with $j=1,\ldots,K-1$, with eigenfunctions $\varphi^R_j$. 
\item For $t<0$, the $K-$th eigenvalue of $\Omega^0(t)$ is $(1-t)^{-2}\lambda_1^L$ with eigenfunction $h_\sharp(t)\varphi^L_1$ and the $(K+1)-$th eigenvalue is $\lambda^R_K$ with eigenfunction $\varphi^R_K$. For $t>0$, the situation is symmetric with $\lambda^R_K<(1-t)^{-2}\lambda_1^L$.
\item At $t=0$, the $K-$th eigenvalue is the double one created by the crossing of the spectral curves above. 
\end{enum2}

\vspace{3mm}

To complete this non-connected domain $\Omega^0_0$, we add a small channel $\omega^\eta$ connecting smoothly its left and right parts. 
The parameter $\eta$ belongs to $(0,1]$ and describes the thickness of the channel. We set $\Omega^\eta_0=\Omega^0_0\,\cup\,\omega^\eta$ and we assume that it is diffeomorphic to the reference domain of Theorem \ref{th_main} (the connection with $\omega^\eta$ is smooth, but it does not remove the possible corners and holes of the starting domain). When $\eta$ goes to zero, the domain $\Omega^\eta_0$ converges to $\Omega^0_0$ in a singular way, as it is classical for the dumbbell shaped domains. More precisely, we need that the spectrum of the Dirichlet Laplacian operator in $\Omega^\eta_0$ converges to the corresponding one in $\Omega^0_0$ in the sense of \cite{Arendt-Daners,Daners}, see also references therein. It is sufficient to satisfy the hypotheses of Theorem \ref{th_cv_spectrum} and it is the case for any natural choice of shape for the thin channels $\omega^\eta$.We refer to Figure \ref{fig-domain} to convince the reader that all the required properties can be satisfied by $\Omega^\eta_0$.    

\vspace{3mm}

Now, we consider the evolution of a solution of the Schr\"odinger equation when we move the domain from $-\tau$ to $\tau$. We denote by $\Omega^\eta(t)$ the domains $h(t,\Omega^\eta_0)$ for $t\in (-\tau,\tau)$ and $\eta\in[0,1]$. For $\eta>0$, these are dumbbell shaped domains. The variation of $t$ changes the size of the left part and also slightly deforms the connecting channel. When $\eta$ goes to $0$, the connecting channel disappears. By abuse of notations, we still denote by $\varphi^R_j$ the extensions by zero of the eigenfunctions of the right part to the whole domain $\Omega^\eta(t)$. We set $\varphi^L(t)=h_\sharp(t)\varphi_1^L$ where we use the same notation again, for the eigenfunction of the left part and its extension. Notice that these extensions by zero still belong to $H^1_0(\Omega^\eta(t))$. Moreover, since we consider only a finite number of energy levels, their $H^1_0-$norms, related to the corresponding eigenvalues, are bounded by a constant $M$, independent of $t\in(-\tau,\tau)$ or $\eta\in(0,1]$. We apply Proposition \ref{prop_difference} with $R=2(\|h\|_{\Cc^3}+\|h^{-1}\|_{\Cc^3})$ and $T=\tau$, and we fix $\tau>0$ small enough such that $2C\sqrt{\tau}M\leq \delta'$. Since the estimation of Proposition \ref{prop_difference} is independent of the domain and thus of $\eta$, it ensures the following property.
\begin{enum2}\setcounter{compte}{3}
\item For all $\eta\in(0,1]$, if $u(t)$ is the solution of the Schr\"odinger equation in the moving domain $\Omega^\eta(t)$ with initial data $u(-\tau)=\varphi^R_j$ with $j\leq K$ or $u(-\tau)=\varphi^L(-\tau)$, then, for all $t\in[-\tau,\tau]$, we have
$$\|u(t)-u(-\tau)\|_{L^2}\leq\delta '.$$ In addition, this property is also true when the motion $h(t)$ is slightly modified or when $\tau$ is smaller.
\end{enum2}
Since $\varphi^L(t)$ is a homothetic transformation of $\varphi_1^L$, up to choose $\tau$ even smaller, we can also assume that $$\|\varphi^L(t)-\varphi^L(-\tau)\|_{L^2}\leq\delta',\ \ \ \ \ \forall t\in[-\tau,\tau].$$ 

\vspace{3mm}

The properties (i)-(iii) hold for the split domain $\Omega^0(t)$. We now choose $\eta>0$ small enough to approximate these properties by the corresponding ones for the domain $\Omega^\eta(t)$. More precisely, for all $j<K$, the $j-$th eigenvalue of the Dirichlet Laplacian operator on $\Omega^0(0)$ is simple, see property (i).
By the convergence of the spectrum recalled in Theorem \ref{th_cv_spectrum}, we can choose $\eta_0>0$ and $\tau>0$ small enough such that 
for all $t\in(-\tau,\tau)$ and $\eta\in[0,\eta_0]$, the $j-$th eigenvalue of  the Dirichlet Laplacian operator on $\Omega^\eta(t)$ is also simple. In addition, we can choose a smooth curve $\varphi_j(t)$ of corresponding eigenfunctions such that
\begin{enum2}\setcounter{compte}{4}
\item For all $j\leq (K-1)$ and $t\in[-\tau,\tau]$, the eigenfunction $\varphi_j(t)$ satisfies the following identity $$\|\varphi_j(t)-\varphi_j^R\|_{L^2}\leq \delta'.$$ 
\end{enum2}
We now fix $\tau>0$ small enough such that (iv) and (v) hold, and we restrict the deformations of the domains $h(t)$ to this time interval. Due to property (ii) above, by choosing $\eta$ smaller if necessary, we can also assume that the $K$ and $K+1$ eigenvalues of the Dirichlet Laplacian operator on $\Omega^\eta(t)$ are also simple at $t=\pm \tau$ and that the corresponding eigenfunctions are close to $\varphi^R_K$ and $\varphi^L(\pm \tau)$. Of course, due to the crossing stated in (iii), we cannot hope to have the convergence for all $t$ between $-\tau$ and $\tau$. However, we can also assume that the two-dimensional spectral projector corresponding to the $K-$th and $(K+1)-$th eigenvalues together are close up to an error $\delta'$. To simplify the notation, as a final adjustment, we allow a small perturbation of $h(t)$ given by Theorem \ref{th_Teytel} such that the $K-$th and $(K+1)-$th eigenvalues of the Laplacian operator on $\Omega^\eta(t)$ are simple for all $t\in(-\tau,\tau)$. We let the reader check that we were careful to make all the arguments above uniform in a small neighborhood of $h$. Due to this simplicity, we can choose smooth curves of eigenfunctions $\varphi_K(t)$ and $\varphi_{K+1}(t)$ corresponding to the   $K-$th and $(K+1)-$th eigenvalues of the Laplacian operator on $\Omega^\eta(t)$ such that the following relations are verified.
\begin{enum2}\setcounter{compte}{5}
\item At $t=-\tau$, we have  $\|\varphi_K(-\tau)-\varphi^L(-\tau)\|_{L^2}\leq \delta'$ and $\|\varphi_{K+1}(-\tau)-\varphi^R_K\|_{L^2}\leq \delta'$.
\item At $t=\tau$, we have $\|\varphi_K(\tau)-\varphi^R_K\|_{L^2}\leq \delta'$ and $\|\varphi_{K+1}(\tau)-\varphi^L(\tau)\|_{L^2}\leq \delta'$.
\end{enum2}
Notice that $\varphi^L(-\tau)$ belongs to the limit eigenspace span$(\varphi^L(t),\varphi^R_K)$ for all 
$t\in[-\tau,\tau]$, up to a small error $\delta'$, see the remark below (iv). Applying the convergence of the 
two-dimensional spectral projector corresponding to the $K-$th and $(K+1)-$th eigenvalues (see Theorem 
\ref{th_cv_spectrum}), we can also ensure the following property up to take a thinner channel $\omega_\eta$.
\begin{enum2}\setcounter{compte}{7}
\item For all $t\in[-\tau,\tau]$, it is satisfied $|\la \varphi^L(-\tau)|\varphi_K(t)\ra|^2 +  |\la \varphi^L(-\tau)|\varphi_{K+1}(t)\ra|^2 =  1 \pm 2\delta'$ where we use the notation $\pm\alpha$ to denote an error term of size at most $\alpha$.   
\end{enum2}

\vspace{3mm}

Now, the global setting is finally defined. It remains to check that it fulfills Lemma \ref{lemma_step}. To recover the notations of its statement, we set $t_0:=-\tau$, $\Omega(t_0)=\Omega^\eta(t_0)$, with $\eta$ as small as required above, and $\varphi_j:=\varphi_j(-\tau)$ for $j\leq K+1$. Let $u_j(t)$ be the solution of the Schr\"odinger equation with the chosen moving domains and with the initial data $u_j(t_0)=\varphi_j$. By linearity, we have $u(t)=\sum c_ju_j(t)$. For all $j<K$, by (iv) and (v), we get
$$\forall t\in [-\tau,\tau]~,~~\|u_j(t)-\varphi_j(t)\|_{L^2}\leq \|u_j(t)-\varphi_j(-\tau)\|_{L^2} +  \|\varphi_j(-\tau)-\varphi_j^R\|_{L^2} + \|\varphi_j^R-\varphi_j(t)\|_{L^2} \leq 3\delta '~.$$
Since, by (vi) and (vii), $\la \varphi_{K}(-\tau)|\varphi^L(-\tau)\ra=1\pm \delta'$
and $\la \varphi_{K}(\tau)|\varphi^L(-\tau)\ra=0\pm \delta'$, there is an intermediate time $t\in[-\tau,\tau]$ such that $|\la \varphi_{K}(t)|\varphi^L(-\tau)\ra|=\alpha\pm \delta'$. Due to (viii), we also have $|\la \varphi_{K+1}(t)|\varphi^L(-\tau)\ra|=\sqrt{1-\alpha^2} \pm 3\delta'$. Using (iv), we obtain
$$ |\la \varphi_{K}(t)|u_K(t)\ra|= |\la \varphi_{K}(t)|\varphi^L(-\tau)\ra|\pm\delta'=\alpha\pm 2\delta'$$
and in the same way 
$$ |\la \varphi_{K+1}(t)|u_K(t)\ra|= \sqrt{1-\alpha^2} \pm 4\delta'~.$$
Thus, at this precise time $t$, we can choose $\theta_K$ and $\theta_{K+1}$ in $\RR$ such that
$$\Big\|u(t)-\big(\alpha e^{i\theta_K}c_K\varphi_K(t)+\sqrt{1-\alpha^2}e^{i\theta_{K+1}}c_K\varphi_{K+1}(t) + \sum_{j=1}^{K-1}c_j\varphi_j(t)\big)\Big\|_{L^2}\leq 4 (K+1)(\max |c_j|)\delta'=\delta.$$
Set $t_1:=t$ and $\Omega(t_1):=\Omega^\eta(t)$. It simply remains to notice that, due to the simplicity of the spectrum, the eigenfunctions $\psi_j$ correspond to $\varphi_j(t)$ up to a phase shift $\theta_j$.
\end{proof}

\vspace{3mm}

\noindent{\bf Remark:} The strategy behind Lemma \ref{lemma_step} is robust and can be performed in different ways. For example, we may consider other type of domains than the dumbbell shaped one. We can also allow crossings of $\lambda^\eta_K(t)$ and $\lambda^\eta_{K+1}(t)$ since the arguments should still hold once we are able to define continuous branches of eigenfunctions $\varphi_K(t)$ and $\varphi_{K+1}(t)$ satisfying the exchange stated in (vi) and (vii). As a different strategy, we can replace, in some situations, the brief deformation in the small interval $(-\tau,\tau)$ by a slow and long adiabatic deformation or use conical intersections as in \cite{ugo1}. We refer to Section \ref{section_examples} for further discussions.

\subsection{Adjusting the phases}\label{rotations}

The arguments above allow us to distribute the energy between the different eigenmodes. However, to drive the solution close to a given state, we also need to produce the correct phases. This is a classical issue with a classical simple solution. 
\begin{lemma}\label{lemma_phases}
Let $\Omega_0\subset\Bc$ be a domain in which the Dirichlet Laplacian operator has a simple spectrum $(\lambda_j)_{j\in\NN^*}$ with a corresponding Hilbert basis of eigenfunctions $(\varphi_j)_{j\in\NN^*}$. Let $N\geq 1$ and let $u_0=\sum_{j=1}^N c_j \varphi_j$ with $c_j\in\CC$. For any real phases $(\theta_j)_{j=1\ldots N}$ and $\delta>0$, there exists a time $T$ and a motion of domains $h\in\Path^3([0,T],\Bc)$ such that the following holds. We have $h(0)=h(T)=\id$ and if $u(t)$ is the solution of the Schr\"odinger equation in the moving domains $\Omega(t)=h(t,\Omega_0)$, then 
$$\Big\|u(T)-\sum_{j=1}^N c_je^{i\theta_j}\varphi_j\Big\|_{L^2}\leq \delta~.$$
\end{lemma}
\begin{proof}
Due to the generic rational independence stated in Theorem \ref{th_irrational}, we can find a domain $\Omega_1$ diffeomorphic to $\Omega_0$ with a corresponding Laplacian operator having rationally independent eigenvalues. Denote by $(\mu_j)_{j\in\NN^*}$ this spectrum and by $(\psi_j)_{j\in\NN^*}$ the corresponding eigenfunctions. 
We consider a deformation of the domain going from $\Omega_0$ to $\Omega_1$, staying equal to $\Omega_1$ for a short time and then going back to $\Omega_0$. We would like that the corresponding solution $u(t)$ has the same distribution of energy on the eigenmodes at the beginning and at the end up to a small error. This is possible, either by choosing $\Omega_1$ very close to $\Omega_0$ and using the continuity stated in Proposition \ref{prop_difference}, or by traveling from both domains very slowly and using the adiabatic result stated in Lemma \ref{lemma_adiab} below. Once we know how this back-and-forth motion modifies the phases, we can stop the dynamics at $\Omega_1$ for longer time. Here, the solution evolves as $\sum_{j=1}^N c_j e^{i(\mu_j t+\alpha_j)}\psi_j$. Due to the rational independence of the $\mu_j$'s, the trajectory $t\mapsto  (\mu_j t+\alpha_j)_{j=1\ldots N}\in\TT^N$ is dense in the torus and we can find a time such that $u(t)$ has the suitable phases, up to a small error. Going back to $\Omega_0$ changes these phases but in a way that has been anticipated. 
\end{proof}

\subsection{Proof of Claim \ref{claim_main}}

In this Section, we complete the proof of Theorem \ref{th_main} by combining the arguments above in order to prove Claim \ref{claim_main}. It could be useful to keep in mind the insight provided by Section \ref{section_strategy}, Figure \ref{fig-spectrum}, and Figure \ref{fig-domain}.

\vspace{3mm}

Fix an error $\delta>0$ equal to $\varepsilon/(2N)$. Using Lemma \ref{lemma_phases}, we know that the problem of the phases can be repaired at the end, up to an error $\delta$. To simplify the notations, from now on, a state will be represented by its distribution of energy when it will be defined on a domain $\Omega$ where the spectrum of the associated Laplacian operator is simple. In other words, $(a_1,a_2,\ldots,a_N)$ stands for a state $\sum_{j=1}^N e^{i\theta_j} a_j\varphi_j$ where $\varphi_j$ is a Hilbert basis of eigenfunctions of $L^2(\Omega)$ and $\theta_j\in\RR$. Following this convention, we start with the state $(A,0,0,\ldots)$ in the domain $\Omega'_0\subset\Bc\subset \RR^d$. We denote by $a_j=|b_j|$ the coefficients of the aimed distribution of energy when the phases are neglected.

\vspace{3mm}

By Lemma \ref{lemma_step}, there exists a deformation between two domains $\Omega_1$ and $\Omega_2$ that drives the state $(A,0,\ldots)$ to $(a_1,\sqrt{A^2-a_1^2},0,\ldots)$ up to an error $\delta>0$. To go from our first domain $\Omega'_0$ to $\Omega_1$ without changing the distribution of energy, we use an adiabatic motion as given by the following result.
\begin{lemma}\label{lemma_adiab}
Let $\Omega_0$ and $\Omega_1\subset\Bc$ be two homotopic domains in which the spectrum of the Dirichlet Laplacian operator is simple. Denote by $(\varphi_j)_{j\in\NN^*}$ and $(\psi_j)_{j\in\NN^*}$ two respective Hilbert basis of eigenfunctions in these domains. Let $u_0=\sum_{j=1}^N c_j\varphi_j$ be given. For all $\delta>0$, there exist $T>0$ and a deformation of domains $h\in\Path^3([0,T],\Bc)$ such that $h(0,\Omega_0)=\Omega_0$ and $h(T,\Omega_0)=\Omega_1$. Moreover, if $u(t)$ is the solution of the Schr\"odinger equation in $\Omega(t)=h(t,\Omega_0)$ with initial data $u(0)=u_0\in L^2(\Omega_0)$, then there exist $\theta_j\in\RR$ such that  
$$\Big\|u(T)-\sum_{j=1}^N c_je^{i\theta_j} \psi_j\Big\|_{L^2(\Omega_1)}\leq\delta~.$$
\end{lemma}
\begin{proof}
It is sufficient to combine the existence of a path avoiding multiples eigenvalues as given by Theorem \ref{th_Teytel} and an adiabatic dynamics as given by Proposition \ref{prop_adiabatic}.   
\end{proof}

First, we use Lemma \ref{lemma_adiab} to go from $(A,0,0,\ldots)$ in $\Omega'_0$ to $(A,0,0,\ldots)$ in $\Omega_1$ up to an error $\delta$. Second, we use Lemma \ref{lemma_step} to arrive in $\Omega_2$ with the state $(a_1,\sqrt{A^2-a_1^2},0,\ldots)$ up to an error $2\delta>0$.  Third, Lemma \ref{lemma_step} provides a deformation between two new domains $\Omega_3$ and $\Omega_4$ driving the state $(a_1,\sqrt{A^2-a_1^2},0,\ldots)$ to the state $(a_1,a_2,\sqrt{A^2-a_1^2-a_2^2})$ up to an error $\delta>0$. Hence, starting with our state $(a_1,\sqrt{A^2-a_1^2},0,\ldots)$ in $\Omega_2$ (up to an error $2\delta$), we use Lemma \ref{lemma_adiab} to drive the state to $(a_1,\sqrt{A^2-a_1^2},0,\ldots)$ in $\Omega_3$ (up to an error $3\delta$). Then, we use the foreseen application of Lemma \ref{lemma_step} to obtain the state $(a_1,a_2,\sqrt{A^2-a_1^2-a_2^2})$ in $\Omega_4$ up to an error $4\delta>0$\ldots

We use this argument iteratively: Lemma \ref{lemma_step} constructs the distribution of the energy level by level and Lemma \ref{lemma_adiab} enables to travel between the different domains required by Lemma \ref{lemma_step}. After $(N-1)$ repetitions of this strategy, we obtain the distribution of energy $(a_1,\ldots,a_N)$ in a domain $\Omega_{2N-2}$ up to an error $(2N-2)\delta$. Then, we use an adiabatic motion of Lemma \ref{lemma_adiab} in order to go back to the initial domain $\Omega'_0$ with a distribution $(a_1,\ldots,a_N)$ up to an error $(2N-1)\delta$. Finally, it remains to use Lemma \ref{lemma_phases} to adjust the phases and to obtain the precise state $\sum_{j=1}^N b_j \varphi_j$ up to an error $2N\delta=\varepsilon$.


\section{Study of a particular example: rectangular and quasi-rec\-tan\-gu\-lar domains}\label{rectangle}

In this section, we explore the case of rectangles in $\RR^2$ with moving boundaries. In this particular framework, several arguments of our strategy can be made more explicit. To follow a motion where the spectrum of the Laplacian operator stay simple, as in Section \ref{section_simplicity}, we need to leave the family of rectangular domains. However, the rectangular shape makes Hypothesis (SAH2) easy to check and we can find explicitly the perturbations enabling to break the double eigenvalues.

In this section, we even design a control by a deformation different from the dumbbell shape adopted in the proof of Theorem \ref{th_main}. It shows how our general arguments are robust and may be applied for various control strategies. We consider the two-dimensional framework, but the results of this section can be easily extended to the multi-dimensional case. 

\subsection{Rectangular domains: the basic motion}
We act on the quantum state by moving the sizes of a family of rectangles 
\begin{equation}\label{mov_rect}
\Omega(t)=\big(0,f_1(t)\big)\times \big(0,f_2(t)\big).
\end{equation}
where $f_i \in\Cc^2([0,T],\RR^+)$ with $j=1,2$. If we need to transport \eqref{SE} from $\Omega(t)$ to a fixed domain, the simplest way is as follows. Consider the square $\Omega_0=(0,1)\times(0,1)\subset\RR^2$ as the reference domain. We set $h(t):~y=(y_1,y_2) \in \Omega_0~\longmapsto~\big(f_1(t)\,y_1,f_2(t)\,y_2\big)\in \Omega(t)$ which defines a family of diffeomorphisms $h(t)$ such that $\Omega(t)=h(t,\Omega_0)$. As presented in \cite[Section 5.2]{SEmoving}, instead of directly using the transformations of Section \ref{section_moving}, it is simpler to  perform a gauge transformation.  We set $\psi(t,x)=\frac{1}{4}\left(\frac{f_1'(t)}{f_1(t)}\,x_1^2+\frac{f_2'(t)}{f_2(t)}\,x_2^2\right)$. Then $u$ solves \eqref{SE} if and only if $w=h^\sharp e^{-i\psi}u$ satisfies the equation
\begin{equation}\label{SE_square1}
i\partial_t w~=~- \frac{1}{f_1(t)^2}\ddd_{y_1y_1}^2 w- \frac{1}{f_2(t)^2}\ddd_{y_2y_2}^2 w ~+~ \frac{1}{4}\left( f''_1(t) 
f_1(t) y_1^2+f''_2(t) 
f_2(t) y_2^2 \right) w.
\end{equation}

\vspace{3mm}

One of the useful features of rectangular domains is the fact that the spectrum is completely known and the eigenmodes are decoupled. More precisely, we set 
\begin{align}\label{modes}\phi_{k_1,k_2}(t) =\varphi^1_{k_1}(t,x_1)\varphi^2_{k_2}(t,x_2) \ \ \ \ \ \ \  \ \ \ \text{with} \ \ \ \ \varphi^j_k(t,\cdot)=\sqrt{\frac{2}{f_j(t)}}\sin\Big(\frac{k \pi}{f_j(t)} \cdot \Big)~.\end{align}
The eigenmode $\phi_{k_1,k_2}$ corresponds to the eigenvalue $\lambda_{k_1,k_2}=\frac{\pi^2 k_1^2}{f_1^2}+\frac{\pi^2 k_2^2}{f_2^2}$ .  
Notice that these eigenvalues may intersect, but always in a smooth way in the sense that the spectral projection on $\phi_{k_1,k_2}(t)$ depends smoothly of $t$. In this case, it is expected that adiabatic theory applies as in the case of a simple spectrum considered in Proposition \ref{prop_adiabatic}.
\begin{prop}[Adiabatic motion of rectangles]\label{prop_adiabatic_square}$~$\\
Let $\Omega(\tau)$ be a family of rectangles defined as above with $f_1,f_2\in \Cc^2([0,1],\Omega_0)$. For every family of eigenmodes $\phi_{k_1,k_2}(\tau)$, $k_j\in\NN^*$ and every $u_0\in L^2(\Omega(0))$, we have
$$\big|\la u_\epsilon(1/\epsilon)|\phi_{k_1,k_2}(1/\epsilon)\ra \big|~=~\big|\la u_0|\phi_{k_1,k_2}(0)\ra\big|+\Oc_{\epsilon\rightarrow 0}(\epsilon)\,,$$
where $u_\epsilon(t)$ is the solution of \eqref{SE_adiabatic} with initial state $u_0$. 
\end{prop}
In particular, an adiabatic deformation of the rectangle drives $\phi_{k_1,k_2}(0)$ close to the mode $\phi_{k_1,k_2}(1)$. The ordering of the sequence of eigenvalues $\lambda_{k_1,k_2}$ depends on the lengths $f_j(\tau)$, and the rank of $\phi_{k_1,k_2}(1)$ might not be the same as the one of $\phi_{k_1,k_2}(0)$. In other words, the adiabatic deformation of the rectangle allows for passing through the eigenvalue crossings that appear during the deformation of the rectangle and for performing the permutations of eigenmodes.\\ 

{ \noindent \emph{\textbf{Proof of Proposition \ref{prop_adiabatic_square}:}}} We did not find any accurate reference for a  version of the adiabatic theorem with crossing of eigenvalues, which directly applies in the general situation of a moving domain $\Omega(\tau)$ (the corresponding Hamiltonian depends on the time in a not so classical way). However, the proof in the case of rectangular shapes is not difficult (in particular thanks to the gauge transformation, which is not always possible, see \cite{SEmoving}). We provide it below for the sake of completeness.

We adapt the gauge transform above: we set $\psi_\epsilon(\tau,x)=\frac{\epsilon}{4}\left(\frac{f_1'(\tau)}{f_1(\tau)}\,x_1^2+\frac{f_2'(\tau)}{f_2(\tau)}\,x_2^2\right)$ and $w_\epsilon(t)=h^\sharp(\epsilon t) e^{-i\psi_\epsilon(\epsilon t)}u_\epsilon(t)$. Again, $u_\epsilon(t)$ is the solution of \eqref{SE_adiabatic} if and only if $w_\epsilon$ satisfies the equation
\begin{equation}\label{SE_square_adiab}
i\partial_t w_\epsilon(t)~=~- \frac{1}{f_1(\epsilon t)^2}\ddd_{y_1y_1}^2 w_\epsilon- \frac{1}{f_2(\epsilon t)^2}\ddd_{y_2y_2}^2 w_\epsilon ~+~ \frac{\epsilon^2}{4}\left( f''_1(\epsilon t) f_1(\epsilon t) y_1^2+f''_2(\epsilon t) f_2(\epsilon t) y_2^2 \right) w_\epsilon
\end{equation}
in the square $(0,1)^2$. Up to a mutiplicative constant, the eigenmode $\phi_{k_1,k_2}(x)$ in $\Omega(\tau)$ becomes $\psi(y):=\sin(k_1\pi y_1)\sin(k_2\pi y_2)$ in the square and is independent of $t$. We also notice that the phase $e^{-i\psi_\epsilon(\epsilon t)}$ behaves as $1+\Oc(\epsilon)$ for small $\epsilon>0$. Thus, proving Proposition \ref{prop_adiabatic_square} comes down to show that 
$$ \big|\la w_\epsilon(1/\epsilon)|\psi \ra \big| = \big|\la w_\epsilon(0)|\psi\ra\big| + \Oc_{\epsilon\rightarrow 0}(\epsilon)$$
for any solution of \eqref{SE_square_adiab}. First notice that, due to Hamiltonian structure of \eqref{SE_square_adiab}, the $L^2-$norm of $w_\epsilon$ is constant and the last term of \eqref{SE_square_adiab} is of order $\Oc(\epsilon^2)$. Thus, we get that
\begin{align*}
\partial_t  \la w_\epsilon(t)|\psi \ra ~&=~ i \left\la \Big(\frac{1}{f_1(\epsilon t)^2}\ddd_{y_1y_1}^2 + \frac{1}{f_2(\epsilon t)^2}\ddd_{y_2y_2}^2\Big) w_\epsilon \,\Big|\, \psi \right\ra ~+~\Oc(\epsilon^2)\\
&=~ i \left\la  w_\epsilon \,\Big|\, \Big(\frac{1}{f_1(\epsilon t)^2}\ddd_{y_1y_1}^2 + \frac{1}{f_2(\epsilon t)^2}\ddd_{y_2y_2}^2\Big) \psi \right\ra~+~\Oc(\epsilon^2)\\
&= ~ -i \Big(\frac{\pi^2 k_1^2}{f_1(\epsilon t)^2}+\frac{\pi^2 k_2^2}{f_2(\epsilon t)^2} \Big) \la w_\epsilon(t)|\psi \ra  ~+~\Oc(\epsilon^2)\\
&:=~-i\lambda_{k_1,k_2}(\epsilon t)  \la w_\epsilon(t)|\psi \ra  ~+~\Oc(\epsilon^2) ~.
\end{align*}
Thus, if we set $\Lambda(\tau)=\int_0^\tau \lambda_{k_1,k_2}(\sigma)\,{\rm d}\sigma$, then we obtain
$$\la w_\epsilon(1/\epsilon)|\psi \ra~=~ e^{-i\Lambda(1)/\epsilon} \la w_\epsilon(0)|\psi \ra~+~\Oc(\epsilon)~.$$
This concludes the proof of Proposition \ref{prop_adiabatic_square} (even providing the exact phase shift).
{\hfill$\square$\\}

\subsection{Decoupling: application of the 1D bilinear control}\label{bilinear_control}

The main feature of the family of rectangular domains is the possibility of decoupling the horizontal and vertical coordinates. 
If $u$ is decomposed as $u(t,x_1,x_2)=u_1(t,x_1)u_2(t,x_2)$ for $(x_1,x_2)\in\Omega(t)$, then $w=h^\sharp e^{-i\psi}u$ is also a product of functions 
\begin{equation}\label{separation}w(t,y_1,y_2)=w_1(t,y_1)w_2(t,y_2) \ \ \text{with}\ \ \ w_j(t,y_j)=\frac 1{\sqrt{f_j}} e^{-\frac i4 f'_j f_j y_j^2} u_j(t,f_j(t)y_j).
\end{equation}
It is straightforward to split \eqref{SE_square1} and to check that $w_1$ and $w_2$ are respectively solutions of the following equations
\begin{equation}\begin{split}\label{SE_square2}
i\partial_t w_1~=~- \frac{1}{f_1(t)^2}\ddd_{y_1y_1}^2 w_1~+~ \frac{1}{4} f''_1(t) 
f_1(t) y_1^2 w_1 \ \ \ \  \ \text{in}\ \ \ \ (0,1),\\
i\partial_t w_2~=~- \frac{1}{f_2(t)^2}\ddd_{y_2y_2}^2 w_2 ~+~ \frac{1}{4}f''_2(t) 
f_2(t) y_2^2  w_2 \ \ \ \  \ \text{in}\ \ \ \ (0,1).
\end{split}
\end{equation}
Vice-versa, if $w_1$ and $w_2$ are solutions of \eqref{SE_square2}, then $w=w_1w_2$ is solution of \eqref{SE_square1} which provides a solution $u$ of \eqref{SE}. We can exploit this decoupling as follows.
First, we simplify the expressions in \eqref{SE_square2} in order to eliminate the time-dependence of the main operator. We use the following change of variables appearing in \cite{Beauchard,Moyano} (see also \cite{BMT,Beauchard_Teismann,Rouchon}): 
\begin{equation}\label{eq_tau}
\tau_j=\int_0^t \frac{1}{f_j(s)^2}ds \ \ \ \text{and}\ \ \ U_j(\tau_j)=\frac{f_j'(t)f_j(t)}{4}~~~~j=1,2.
\end{equation}

We substitute these elements in the corresponding equation in \eqref{SE_square2}. We obtain two completely decoupled bilinear Schr\"odinger equations
\begin{equation}\label{SE_square3}
i\partial_{\tau_j} w_j~=~- \ddd_{y_jy_j}^2 w_j~+~ \Big(U_j'(\tau_j)-4U_j(\tau_j)\Big)y_j^2 w_j \ \ \ \  \ \text{in}\ \ \ \ (0,1)~,~~ j=1,2.
\end{equation}
We can now use the well-known results concerning the bilinear control of the one-dimensional Schr\"odinger equation to obtain the following control. Notice that similar approaches for the one-dimensional case are used in \cite{Beauchard,Beauchard_Teismann,Moyano,Rouchon}.

\begin{prop}[Approximate control for decoupled data]\label{control_bilinear_decoupled}$~$\\
Let $\Omega^\ii=(0,a)\times(0,b)$ with  $a,b>0$. Let $u^\ii, u^\ff\in L^2(\Omega^\ii)$ satisfying $\|u^\ii\|_{L^2}=\|u^\ff\|_{L^2}$ and admitting a decoupling 
$$u^\ii(y_1,y_2)=u_1^\ii(y_1)u_2^\ii(y_2) ~~~\text{ and }~~~ u^\ff(y_1,y_2)=u_1^\ff(y_1)u_2^\ff(y_2)$$  
For every $\varepsilon>0$, there exist $T>0$ and a family of moving rectangles $\{\Omega(t)\}_{t\in (0,T)}$ as in \eqref{mov_rect} such that $\Omega(0)=\Omega(T)=\Omega^\ii$ and such that the solution of the corresponding dynamics \eqref{SE} with initial data $u(t=0)=u^\ii$ satisfies $$\|u(t=T)-u^\ff\|_{L^2}\leq \varepsilon.$$  
 \end{prop}
\begin{proof}
Up to rescaling the components, we can assume that $\|u^\ii_j\|_{L^2}=\|u^\ff_j\|_{L^2}=1$ for $j=1,2$. To prove the result, it is sufficient to provide deformations of the lengths $f_j(t)$ enabling to control the equations \eqref{SE_square3}.
Due to the decoupling, we can apply the one-dimensional results. The bilinear equation
\begin{equation}\label{eq_cbd1}
i\ddd_\tau w=-\ddd_{yy}^2 w + V(\tau) y^2 w~~~~~~~~y\in (0,1)
\end{equation}
is globally approximately controllable in $L^2(0,1)$ (and in $H^3$) thanks to  \cite[Example 2.2\ \&\ Theorem 4.4]{SE_bili_simu} (see also \cite{chambrion1,nabile,chambrion,SE_bili_stima,milo,nerse2}). Notice that the cited result is stated with $V\in L^2$, but the control can actually be in $\Cc^\infty$. The controllability ensures the existence of two times $\tau^\ff_j>0$, $j=1,2$ and two controls $V_j(\tau)\in \Cc^\infty(0,\tau^\ff_j)$ such that the dynamics of $i\ddd_\tau w_j=-\ddd_{y_jy_j}^2 w_j + V_j(\tau) y_j^2 w_j$ steers $w^\ii_j$ close to $w^\ff_j$ with respect to the $L^2-$norm in a time $\tau^\ff_j$. Moreover, the gauge transformation \eqref{separation} is independent of $f_j$ when the rectangle is not moved, thus  $w^\ii_j$ and $w^\ff_j$ are determined explicitly from $u^\ii$ and $u^\ff$.

The main problem here is to construct functions $f_j(t)$ providing the aimed controls $V_j(\tau)$. To simplify the notations, we omit the index $j$ in this part. We choose a solution $U(t)$ of 
$$ U'(\tau)= 4 U(\tau) + V(\tau) ~~~~~\text{ for } \tau\in [0,\tau^\ff]$$
with an initial data $U(0)$ sufficiently large such that $U(\tau)>0$ for $\tau\in [0,\tau^\ff]$. Since the above ODE is linear, this step is easy. Finding $f$ satisfying \eqref{eq_tau} is instead more subtle. We consider $\tau(t)$ a local solution of the nonlinear second order ODE
\begin{equation}\label{eq_tau2}
\left\{\begin{array}{ll} \tau''(t)&=-8(\tau'(t))^2 U(\tau(t)), \\ \tau(0)&=0,\\ \tau'(0)&=1/a^2.\end{array}\right.
\end{equation}
The theorem of Cauchy-Lipschitz obviously applies to \eqref{eq_tau2} but it provides only the local existence and uniqueness. The solution $\tau(t)$ of \eqref{eq_tau2} exists until $\tau(t)$ leaves $[0,\tau^\ff]$ or until $\tau'(t)$ blows up. We first notice that $\tau''(t)<0$ and thus $\tau'(t)$ is decreasing. Moreover, $(c,0)$ is a solution of \eqref{eq_tau2} for all $c\in [0,\tau^\ff]$. Since $\tau'(0)\neq 0$, $\tau'(t)$ never vanishes. Both remarks imply that $\tau'(t)$ stays in $(0,1/a^2]$ and thus it can not blow up. Since $\tau(t)$ is increasing, there are two possibilities: either the solution $\tau$ exists for all $t>0$, or $\tau(t)$ reaches $\tau^\ff$ at a time $T$. This last possibility is the one we need to control \eqref{eq_cbd1} since $\tau$ has to describe the whole interval $[0,\tau^\ff]$. But
$$\frac{\dd}{\dd t} \Big(\frac 1 {\tau'(t)}\Big)= - \frac {\tau''(t)}{(\tau'(t))^2}=8U(\tau(t))\in \big(0,8\|U\|_\infty \big]$$
Thus, $1/\tau'(t)\leq 1/a^2+8\|U\|_\infty t$ and $\tau(t)\geq \frac 1{8\|U\|_\infty} \Big(\ln\big(1/a^2+8\|U\|_\infty t\big) + 2\ln a \Big)$, which imply that $\tau$ must reach $\tau^\ff$ in a finite time $T$. We can finally set $f(t)=1/\sqrt{\tau'(t)}$ for all $t\in [0,T]$ (recall that $\tau'(t)>0$). It is then straightforward to check that, by construction, the change of variables \eqref{eq_tau} effectively transforms \eqref{SE_square3} to \eqref{eq_cbd1} with $V(\tau)$ the suitable control. 

The above arguments show that we can drive $u^\ii_1$ to $u^\ff_1$ in a time $T_1$ and $u^\ii_2$ to $u^\ff_2$ in a time $T_2$. A priori $T_1\neq T_2$ and we let one of the components evolves following the free Schr\"odinger equation during the time $|T_1-T_2|$. Moreover, it is possible that the final rectangle has not the dimensions $a\times b$. In this case, we deform adiabatically the rectangle to obtain the aimed dimensions. Both associated evolutions do not change the distribution of energy of the modes, but add phases. Thus, we obtain a state $\sum_{k} c_{k} e^{i\theta_{k}} \varphi_{k}$ instead of $u^\ff=\sum_{k} c_{k}  \varphi_{k}$ (where $(\varphi_k)_{k\in\NN}$ denotes a Hilbert basis of eigenfunctions). 
Since we are interested in approximate controllability, it is sufficient to consider a finite sum on the first eigenfunctions $\varphi_k$, $k=0,\ldots, N$. If $a$ and $b$ are rationally independent, then it is sufficient to wait and let the evolution of the free Schr\"odinger equation unfolds the considered first phases up to a sufficiently small error. When $a$ and $b$ are not rationally independent, we firstly deform adiabatically the initial rectangle $\Omega_0$ into a new one $\tilde\Omega_0$ satisfying such hypothesis. Now, we rotate the states as in the previous point and we finally come back adiabatically to $\Omega_0$. The adiabatic back-and-forth deformations of $\Omega_0$ in $\tilde\Omega_0$ also add some phases but we can program the intermediate rotations in order to also remove these new phases (see the arguments of Lemma \ref{lemma_phases}). \end{proof}

Notice that the controllability result from Proposition \ref{control_bilinear_rec} is only valid for the very specific class of states which are separable in the variables. In Section \ref{section_global_rect}, we use this specific control to obtain the global approximate controllability for general quantum states defined on a rectangle.

\subsection{Breaking symmetries: adiabatic motions without crossing of eigenvalues}\label{section_break}

In this subsection, we investigate the existence of deformations of a rectangle in another one, avoiding all the possible crossings of the first $N$ modes. We can preserve the rectangular shape of the domain as in \eqref{mov_rect} as soon as the first $N$ eigenmodes are simple. Each time we approach the shape of a rectangle 
admitting a double eigenvalue, we need to break the rectangular structure with a short deformation given by Theorem \ref{th_Teytel_bis} in order to avoid it. After, we come back to the rectangular shape and we iterate this process until we reach the final domain. The key point here is the use of Theorem \ref{th_Teytel_bis} to preserve the simplicity of the spectrum. Any generic perturbation of the shape would work. However, we show that very specific and simple perturbations are sufficient to break the symmetries. To this end, we need to show the validity of Hypotheses (SAH2) of Section \ref{section_simplicity} along a deformation of rectangular shapes defined as in \eqref{mov_rect}.  

\smallskip

Let $\Omega_0=(0,1)\times(0,1)$ and $h$ be a diffeomorphism $h:(y_1,y_2)\in \Omega_0\mapsto (a y_1,b y_2)$ with $a,b>0$. The spectrum of the Dirichlet Laplacian in $\Omega=h(\Omega_0)$ can present double eigenvalues according to the lengths $a$ and $b$. We consider the parameters $a$ and $b$ so that there is a double eigenvalue
\begin{align}\label{resonance}\lambda=\frac{\pi^2 k_1^2}{a^2}+\frac{\pi^2 k_2^2}{b^2}=\frac{\pi^2 l_1^2}{a^2}+\frac{\pi^2 l_2^2}{b^2}\end{align}
with suitable different $k_1,k_2,l_1,l_2\in\NN^*$. We denote by $\phi_{k_1,k_2}$ and $\phi_{l_1,l_2}$ two corresponding orthonormal eigenfunctions defined as in \eqref{modes}. To by-pass this double eigenvalue, we use 
the strategy of Theorem \ref{th_Teytel_bis} in the class of deformations $\Hc$ consisting of diffeomorphisms of the form $(y_1,y_2)\mapsto (f_1(y_2) y_1,f_2(y_1) y_2)$ with $f_j>0$ polynomials of degree $2$. The class $\Hc$ contains the rectangular deformations as well as tilting or bending of edges (see Table \ref{table1}). Since we argue by locally perturbing a straight path of rectangles, we only need to verify the conditions (SAH2) for the eigenvalue $\lambda$ of the rectangular shape, with perturbations in the tangent space $T_h\Hc$. We compute in Table \ref{table1} the corresponding integrals considered in the conditions (SAH2) by using the notation 
$$I_{m_1,m_2,n_1,n_2}(g)= \Big|\int_{\partial\Omega} \Dnu{\phi_{m_1,m_2}}\Dnu{\phi_{n_1,n_2}} \la h_*g|\nu\ra \dd\sigma\Big|,\ \ \ \ \ \ (m_1,m_2),(n_1,n_2)\in\big\{(k_1,k_2),\ (l_1,l_2)\big\}.$$

\begin{table}[p!]
\begin{center}
\resizebox{0.7\textwidth}{!}{\rotatebox{90}{
\renewcommand{\arraystretch}{1.5}
\begin{tabular}{|M{5cm}|M{3.2cm}|N{2.5cm}|N{2.5cm}|N{4.8cm}N{0pt}|}
\hline
{$h_*g$} 
& Figure
& $I_{k_1,k_2,k_1,k_2}$ 
& $I_{l_1,l_2,l_1,l_2}$ 
& $I_{k_1,k_2,l_1,l_2}$ &  
\rule{0pt}{14mm}\\
\hline
$h_*g_1=(x_1,0)$ 
&\vspace{2mm}\includegraphics[height=12mm]{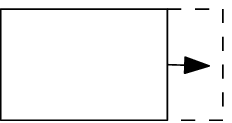} 
& $\frac{2k^2_1\pi^2}{a^3}$ 
& $\frac{2l^2_1\pi^2}{a^3}$ 
&  0 &  \rule{0pt}{14mm}\\ 
\hline
$h_*g_2=(0,x_2)$ 
&\vspace{2mm}\includegraphics[height=12mm]{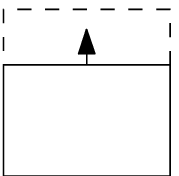}
& $\frac{2k^2_2\pi^2}{b^3}$ 
& $\frac{2l^2_2\pi^2}{b^3}$ 
& 0 & \rule{0pt}{14mm}\\ 
\hline
$h_*g_3= \Big(\big( \frac{x_2}{b}-\frac{1}{2}\big)x_1,0\Big)$
& \vspace{2mm}\includegraphics[height=14mm]{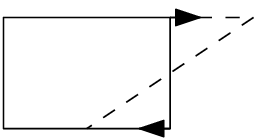} 
& 0 
& 0 
& $\left\{ \begin{array}{l}
            0 \hspace{1cm} \text{{\normalsize{}if $k_2\equiv l_2\text{ mod }2$}}\\ 
            \frac{16b \, k_1l_1 k_2 l_2}{a^3(k_2^2-l_2^2)^2}~~\text{{\normalsize{}otherwise}} 
           \end{array}\right.$ & \rule{0pt}{20mm}\\ 
\hline
$h_*g_4= \Big(  \frac{x_2}{b}\big(1-\frac{x_2}{b}\big)x_1,0\Big)$
& \vspace{2mm}\includegraphics[height=14mm]{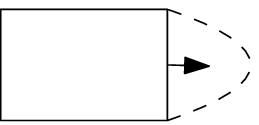} 
& $\frac{ k^2_1(k_2^2\pi^2+3)}{3a^3k_2^2}$
& $\frac{ l^2_1(l_2^2\pi^2+3)}{3a^3l_2^2}$ 
& $\left\{ \begin{array}{l} 
          \frac{16b^2\, k_1l_1 k_2 l_2}{a^3(k_2^2-l_2^2)^2}~~\text{{\normalsize{}otherwise}}\\ 
          0 \hspace{1cm} \text{{\normalsize{}if $k_2\not\equiv l_2\text{ mod }2$}}
           \end{array}\right.$ & \rule{0pt}{20mm}\\ 
\hline
$h_*g_5= \Big( \frac{x_2^2}{b^2} x_1,0\Big)$ 
& \vspace{2mm}\includegraphics[height=12mm]{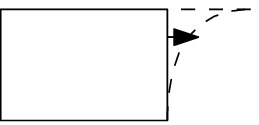} 
& $\frac{k^2_1(2k^2_2 \pi^2-3)}{3 k_2^2 a^3}$
& $\frac{ l^2_1(2l^2_2 \pi^2-3)}{3 l_2^2 a^3}$ 
& $\frac{16 k_1l_1 k_2 l_2}{a^3(k_2^2-l_2^2)^2}$& \rule{0pt}{14mm}\\ 
\hline
\end{tabular}
}}
\end{center}
\caption{\it The family of rectangular deformations is too symmetric to make possible to avoid double eigenvalues. We look for additional deformations enabling to by-pass the double eigenvalues. In this table, we consider several simple perturbations $h^*g$ of a rectangle $(0,a)\times(0,b)$ and we compute the integrals required to check Hypothesis (SAH2) of Section \ref{section_simplicity}.}\label{table1}
\end{table}

Let us consider now the expressions provided in Table \ref{table1}. Both ratio $I_{k_1,k_2,k_1,k_2}/I_{l_1,l_2,l_1,l_2}$ associated to $g_1$ and $g_2$ (the rectangular deformations) are different since  
\begin{equation}\label{rel1}
\frac{I_{k_1,k_2,k_1,k_2}}{I_{l_1,l_2,l_1,l_2}}(g_1)=
\frac{2k^2_1\pi^2}{a^3}\frac{a^3}{2l_1^2 \pi^2}=\frac{k^2_1}{l^2_1}\neq \frac{k^2_2}{l^2_2}=\frac{I_{k_1,k_2,k_1,k_2}}{I_{l_1,l_2,l_1,l_2}}(g_2).
\end{equation}
Indeed we must have $\frac{k^2_1}{l^2_1}\neq \frac{k^2_2}{l^2_2}$ otherwise \eqref{resonance} would imply
$(k_1,k_2)=(l_1,l_2)$. The inequality \eqref{rel1} yields that the functionals corresponding to 
$I_{k_1,k_2,k_1,k_2}$ and $I_{l_1,l_2,l_1,l_2}$ (the first two rows of Table \ref{table1}) are linearly independent, even with simple rectangular deformations. In order to show that also the third functional is linearly independent, we need to consider non-rectangular transformations $g_3$ and $g_4$. The first ensures the property for $g\in T_h\Hc\mapsto I_{k_1,k_2,l_1,l_2}$ when $k_2\not\equiv l_2\text{ mod }2$, while the second when $k_2\equiv l_2\text{ mod }2$. The use of the simple deformations $g_1$, $g_2$, $g_3$ and $g_4$ (Table \ref{table1}) is sufficient to ensure the Hypothesis (SAH2). However, it is also possible to combine $g_3$ and $g_4$ in the transformation $g_5$ which always works but it does not preserve any symmetry at all.

The linear independence of the three functionals yields the validity of the conditions (SAH2) and allows us to use Theorem \ref{th_Teytel_bis} in order to avoid multiple eigenvalues. 
\begin{prop}\label{prop_Teytel_rect}
Let $N\in\NN$ and $\Omega_0=(0,1)\times (0,1).$ Let $h\in\Path^k[0,1]$ be defined as in \eqref{mov_rect} and represent a family of moving rectangles $\Omega(\tau)=h(\tau,\Omega_0)$ such that 
$$\Omega(0)=(0,a)\times (0,b),\ \ \  \ \ \ \Omega(1)=(0,a')\times (0,b')~.$$ 
For all $\varepsilon>0$, there exists a path $g\in\Path^k[0,1]$ such that:
\begin{itemize2}
\item $g(0)=h(0)$, $g(1)=h(1)$ and for all $\tau\in(0,1)$,  $g(\tau)$ belongs to $\Hc$, i.e. it is a combination of the first fourth transformations $g_1,\ldots,g_4$ of Table \ref{table1},
\item $g$ is a small perturbation of the initial rectangular path that is $\|g-h\|_{\Cc^k([0,1]\Omega_0)}< \varepsilon$
\item the $N$ first eigenvalues of the Dirichlet Laplacian operator $-\Delta$ in $\tilde\Omega(\tau)=g(\tau,\Omega_0)$ are simple for all $\tau\in (0,1)$.
\end{itemize2}
\end{prop}

In our framework, it is noteworthy that $g_2$ can be recovered by $g_1$ composed by a homothety. Since this last transformation preserves the simplicity of the spectrum, we can always make the spectrum simple by perturbing only one edge of a rectangular shape: it is not necessary to deform two sides of the rectangle in order to avoid the eigenvalues crossings. We state this result in the following corollary where we denote by $\tilde \Hc$ the manifold of the diffeomorphisms $(y_1,y_2)\mapsto (f_1(y_2) y_1, y_2)$ with $f_1>0$ a polynomial of degree $2$.

\begin{coro}\label{co_Teytel_rect}
In Proposition \ref{prop_Teytel_rect}, if $b=b'$, we can strengthen the fact that $g$ belongs to $\tilde \Hc$ by constructing $g$ with suitable time-varying coefficients $\alpha,\ \beta$ and $\gamma$ such that
$$g(\tau)~:~ (y_1,y_2)\mapsto \big(\, (\alpha(\tau)+ \beta(\tau)y_2+\gamma(\tau)y_2^2 ) y_1\,,\,y_2\,\big).~$$
\end{coro}

\FloatBarrier

\subsection{Global approximate controllability}\label{section_global_rect}
In this section, we present how to approximately control quantum states defined on a two dimensional rectangle by moving its borders. The control is obtained by coupling the two arguments presented above:
\begin{itemize2}
\item[-] It is possible to drive a decoupled state to another decoupled state by changing the dimensions of the rectangle (Proposition \ref{control_bilinear_decoupled}).
\item[-] Adiabatic motions of the rectangle, including slight deformations of a side, preserve the distribution of the energy (Section \ref{section_break}).
\end{itemize2}
It is noteworthy that the result of this section is a particular case of Theorem \ref{th_main}. However, the strategy of control presented here is different from the one of the proof in Section \ref{section_proof}. It underlines that our arguments are generally robust and they provide useful tools for different situations and aims. 
\begin{prop}\label{control_bilinear_rec}
Let $\Omega^\ii=(0,a)\times(0,b)$ with  $a,b>0$. Let $u^\ii, u^\ff\in L^2(\Omega^\ii)$ satisfy $\|u^i\|_{L^2}=\|u^f\|_{L^2}$. For every $\varepsilon>0$, there exist $T>0$ and a family moving domains $\{\Omega_t\}_{t\in (0,T)}$ such that $\Omega(0)=\Omega(T)=\Omega^\ii$ and such that the solution of the corresponding dynamics \eqref{SE} with initial data $u(t=0)=u^\ii$ satisfies $$\|u(t=T)-u^\ff\|_{L^2}\leq \varepsilon.$$  
 \end{prop}
\begin{proof}
In what follows, we denote by $\varphi_j^{\Omega}$ the $j-$th eigenmode of  Dirichlet Laplacian on domain $\Omega.$  In order to control any couple of states in $L^2(\Omega_0)$, it is sufficient to drive the ground state $\varphi_1^{\Omega_0}$ close to any state with norm $1$. Without loss of the generality, we can assume that the target state $u\in L^2(\Omega_0)$ is a linear combination of a finite number of eigenmodes such that
$$u= \sum_{j=1}^Nc_j\varphi_j^{\Omega_0},\ \ \ \  \ \ \  \ \ \ \ \ \{c_j\}_{j\leq N}\subset\CC,\ \ \ \  \ \ \  \ \ \ \ \ \sum_{j=1}^N|c_j|^2=1.$$
In the first step, we deform adiabatically $\Omega_0$ in a rectangle $\Omega_1=(0,a_1)\times(0,b_1)$. It preserves the energy of the ground state as the first eigenvalue of a Dirichlet Laplacian on a connected domain is always simple. We choose $a_1\gg b_1$ so that the first $N$ modes of the Dirichlet Laplacian in $\Omega_1$ have the form
\begin{equation}\label{cond11}
\varphi_j^{\Omega_1}(x_1,x_2)=\frac{2}{\sqrt{a_1b_1}}\sin\Big(\frac {j \pi}{a_1} x_1\Big)\sin\Big(\frac { \pi}{b_1} x_2\Big),\ \ \ \  \ \ (x_1,x_2)\in\Omega_1,\ \ \ \forall j\leq N.
\end{equation}
This first motion steers $\varphi_1^{\Omega_0}$ close to $e^{i\theta }\varphi_1^{\Omega_1}$.
In the second step, we use Proposition \ref{control_bilinear_rec} to drive 
$$e^{i\theta }\varphi_1^{\Omega_1}= \frac{2}{\sqrt{a_1b_1}} e^{i\theta} \sin\Big(\frac {\pi}{a_1} x_1\Big)\sin\Big(\frac {\pi}{b_1} x_2\Big)$$ 
close to
\begin{align*}
\sum_{j=1}^N c_j e^{i\theta_j}  \varphi_j^{\Omega_1}(x_1,x_2)~&=~\sum_{j=1}^N c_j \frac{2}{\sqrt{a_1b_1}} e^{i\theta_j}  \sin\Big(\frac {j \pi}{a_1} x_1\Big)\sin\Big(\frac {\pi}{b_1} x_2\Big)\\
&=~ \sqrt{\frac{2}{b_1}} \sin\Big(\frac {\pi}{b_1} x_2\Big) \left(\sum_{j=1}^N c_j \sqrt{\frac{2}{a_1}} e^{i\theta_j} \sin\Big(\frac {j \pi}{a_1} x_1\Big) \right)
\end{align*}
where the phases $\{\theta_j\}_{j\leq N}$ will be defined later. Notice that the states above are decoupled so Proposition \ref{control_bilinear_rec} may apply. In fact, we only need to control the horizontal part of the state, so it is mainly a one-dimensional result. The trick is to choose $\Omega_1$ sufficiently close to an horizontal segment to ensure that the relevant modes are all ``horizontal''. Finally, we can deform back $\Omega_1$ in $\Omega_0$ adiabatically by avoiding all the crossing of the $N$ first eigenvalues. This last motion is defined by applying the adiabatic regime to the path provided in Proposition \ref {prop_Teytel_rect}.  
This allows us to preserve the distribution of the energy (up to a small error) but it adds some phases to the modes and we obtain (approximatively) the state 
$$u= \sum_{j=1}^N c_j e^{i\theta_j}e^{i\rho_j} \varphi_j^{\Omega_0} $$
where $\rho_j$ do not depend on $\{\theta_j\}_{j\leq N}$. The values $\{\rho_j\}_{j\leq N}$ are well-known in advance and then we can program $\{\theta_j\}_{j\leq N}$ in order to remove all the phases appearing at the end of the motion.
\end{proof}

\subsection{Examples of applications to simple transformations of the states}\label{section_examples}
In this subsection, we present some explicit examples of controls and permutations of modes due to the techniques developed in this work. In what follows, we denote $\Omega=(0,a)\times(0,b)$ with $a>b>0$.

\vspace{5mm}

\noindent{\bf \underline{Switchin}g\underline{ the }q\underline{uantum numbers}.}\\[1mm]
The following is a completely adiabatic deformation of $\Omega$ steering any mode $\phi_{j,k}$ to $\phi_{k,j}$. First, we deform the dimension of the rectangle $(0,a)\times (0,b)$ into $(0,b)\times (0,a)$ adiabatically. Proposition \ref{prop_adiabatic_square} ensures that we follow the mode $\phi_{j,k}$. Then we simply rotate adiabatically the rectangle $(0,b)\times(0,a)$ in $(0,a)\times(0,b)$ and the state becomes $\phi_{k,j}$. See Figure \ref{ex1}.

\begin{figure}[ht]
\begin{center}
\includegraphics[width=\textwidth]{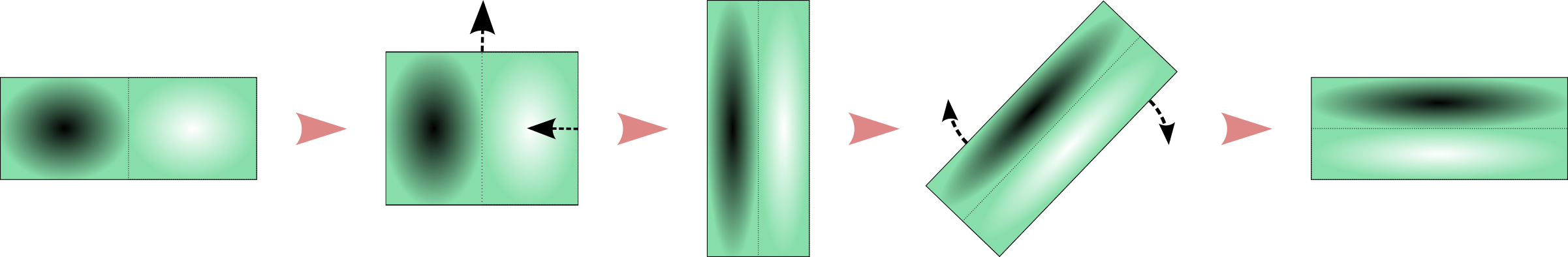}
\end{center}
\caption{\it The figure represents an adiabatic deformation of a rectangle steering $\phi_{2,1}$ in $\phi_{1,2}$.} \label{ex1} 
\end{figure}

\vspace{5mm}

\noindent{\bf \underline{Pum}p\underline{in}g\underline{ an ei}g\underline{enstate into another}.}\\[1mm]
Here, we provide an adiabatic deformation of $(0,a)\times(0,b)$ steering the mode $\phi_{j,k}$ in $\phi_{j',k'}$ as soon as neither mode is the first one. Assume that our aimed state $\phi_{j',k'}$ corresponds to the $p$-th mode. 
\begin{enum2}
\item First, we modify the horizontal edge $(0,a)$ in $(0,a')$ such that in the rectangle $(0,a')\times(0,b)$, the mode $\phi_{j,k}$ corresponds to the $p$-th eigenvalue. This is always possible as soon as $(j,k)\neq (1,1)$ and $p>1$. Proposition \ref{prop_adiabatic_square} shows that, if the motion is sufficiently slow, then we actually drive 
the mode $\phi_{j,k}$ of the rectangle $(0,a)\times(0,b)$ to the mode $\phi_{j,k}$ of the rectangle $(0,a')\times(0,b)$.
\item Second, we deform back the domain in $(0,a)\times(0,b)$ by breaking the rectangular shape of the domain in order to avoid all the eigenvalue crossings. To this purpose, we use Proposition \ref{prop_Teytel_rect} or even Corollary \ref{co_Teytel_rect} and we can stay very close to the family of rectangles of height $b$. Due to Proposition \ref{prop_adiabatic}, the $p$-th mode $\phi_{j,k}$ of the rectangle $(0,a')\times(0,b)$ is transformed into the $p$-th mode of the rectangle $(0,a)\times(0,b)$, which is $\phi_{j',k'}$ by assumption. 
\end{enum2}
Figure \ref{ex2} illustrates the change of $\phi_{2,1}$ into $\phi_{1,2}$. Notice that for concrete applications, when we deform back the domain to the original rectangle, it could be simpler to break the symmetry by adding a generic electric potential rather than tilting or bending the edges. 

\begin{figure}[ht]
\begin{center}
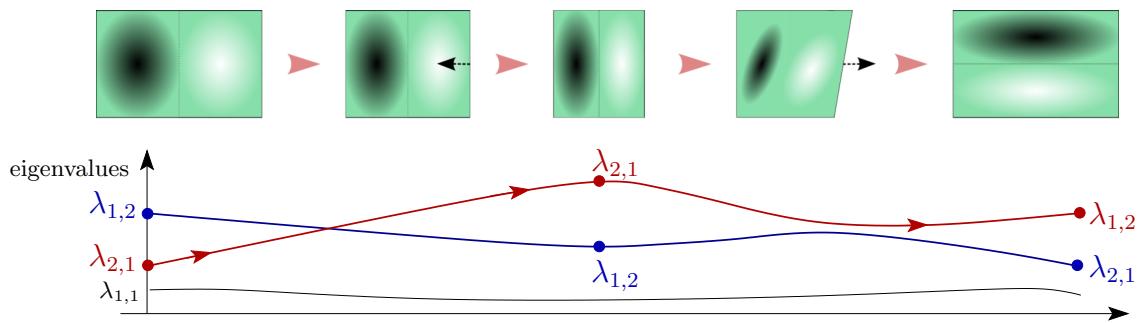
\end{center}
\caption{\it The figure represents a deformation of a rectangle steering $\phi_{2,1}$ in $\phi_{1,2}$. We start with a rectangle $(0,a)\times(0,b)$ with $a>b$ and we reduce the length $a$ to $a'<b$. The energy levels of the states $\phi_{2,1}$ and $\phi_{1,2}$ cross and during an adiabatic motion, the state follows the eigenstate $\phi_{2,1}$. Then, we go back to the original rectangle with a slight deformation of the edge. Here, we break the symmetry allowing to preserves the order of the modes. During an adiabatic motion, the state follows the third eigenstate which is $\phi_{1,2}$ at the end of the motion. Notice that if $a$ and $a'$ are close enough to $b$, then the only possible energy crossing (that we need to avoid when we go back) concerns $\phi_{2,1}$ in $\phi_{1,2}$. Table \ref{table1} shows that a slight tilt of the right edge is sufficient.} \label{ex2} 
\end{figure}

\vspace{5mm}

\noindent{\bf \underline{Creatin}g\underline{ su}p\underline{er}p\underline{osition of excited states}.}\\[1mm]
As we have noticed in the previous examples, different adiabatic deformations of the initial rectangle $\Omega$ yield to different results, depending if we allow the eigenvalues to cross or not.
For example, consider again the mode $\phi_{2,1}$ in a domain $(0,a)\times(0,b)$ where $a>b$. We observe the following phenomena.
\begin{enum2}
\item  If we contract adiabatically the domain by preserving its rectangular shape, then the mode $\phi_{2,1}$ is double when $\Omega$ becomes the square $(0,b)\times (0,b)$. If we continue the reduction slowly to obtain a rectangle $(0,a')\times (0,b)$ with $a'<b$, then the final mode is still $\phi_{2,1}$ due to Proposition \ref{prop_adiabatic_square}. 
\item When we suitably modify this dynamics by breaking the rectangular symmetry of $\Omega$, using Corollary \ref{co_Teytel_rect}, the rank of $\phi_{2,1}$ is preserved and, then, it is steered in the second eigenmode of the new domain, which is not $\phi_{2,1}$ anymore but $\phi_{1,2}$ if $a'<b$.
\end{enum2}
From a spectral point of view, both motions follow the spectral curve associated to the mode $\phi_{2,1}$ until we reach the eigenvalue crossing involving $\phi_{2,1}$ and $\phi_{1,2}$. If we continue to reduce adiabatically the rectangle, then we pass through the crossing by pursuing the mode $\phi_{2,1}$. When we adiabatically modify the shape of $\Omega$ in order to preserve the simplicity of the spectrum, we follow $\phi_{1,2}$ instead. 

Now, assume that we choose an intermediate deformation interpolating the motions (i) and (ii) above. Then, by the intermediate value theorem, it is possible to distribute the initial energy of $\phi_{2,1}$ between the modes $\phi_{2,1}$ and $\phi_{1,2}$, see Figure \ref{ex3}. This idea permits to steer $\phi_{2,1}$ into any superposition of $\phi_{2,1}$ and $\phi_{1,2}$. Notice that the speed of the intermediate motion is obviously not adiabatic for this regime but has been set to be slow enough so that both pure motions (i) and (ii) are adiabatic.
The same technique, applied to a finite number of modes, allows to control any superposition of excited states in any other (similarly to what happens in Figure \ref{fig-spectrum}). Notice that the ground state $\phi_{1,1}$ can not be adiabatically controlled in this way as it is always simple. A possible solution is to deal with the ground state via the control method used in Propositions \ref{control_bilinear_decoupled} and \ref{control_bilinear_rec}.

\begin{figure}[ht]
\begin{center}
\resizebox{13cm}{!}{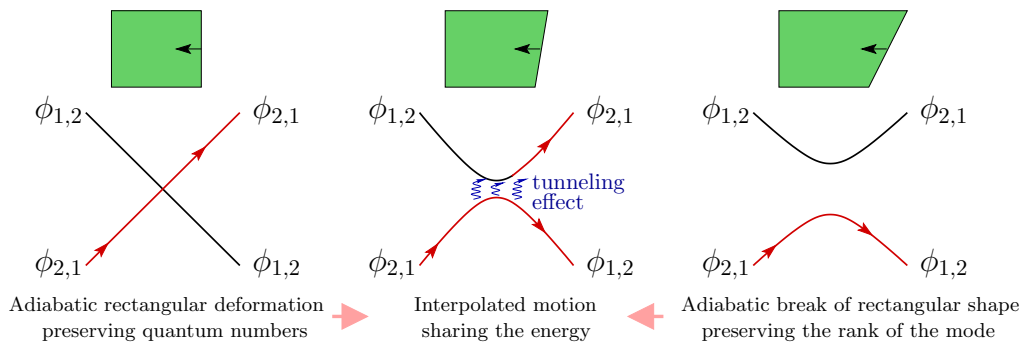}
\end{center}
\caption{\it The figure shows how to use the tunneling effect to create superposed states. On the left, we pass adiabatically an eigenvalue crossing with a rectangular deformation and the quantum numbers of the state do not change thanks to Proposition \ref{prop_adiabatic_square}. On the right, we break the symmetry and there is not any real crossing of eigenvalues anymore. Here, Corollary \ref{co_Teytel_rect} yields that an adiabatic motion preserves the position in the spectrum and it switches the pure states. In the middle, we follow an interpolated deformation with the same speed, which produces an interpolated state. Physically, we use the tunneling effect to create a superposed state.} \label{ex3} 
\end{figure}

Let us also discuss the possibility of using the conical intersections of eigenvalues, following the method introduced in \cite{ugo1}. Consider a family of shapes $\Omega(\alpha,\beta)$ parametrized by two real parameters. Assume that, in the domain $\Omega(0,0)$, the Laplacian operator has an eigenvalue $\lambda_{j}=\lambda_{j+1}$ of multiplicity two. This intersection is conical if the local dependance of the eigenvalues with respect to the parameters satisfies  
\begin{equation}\label{conical}
\exists C,\varepsilon>0~,~~\forall (\alpha,\beta)\in B_{\RR^2}(0,\varepsilon)~,~~\lambda_{j+1}(\alpha,\beta)-\lambda_{j}(\alpha,\beta)\,>\,C(|\alpha|+|\beta|)
\end{equation}
(it is in fact possible to deal with an intersection of more eigenvalues by considering more parameters). In \cite{ugo1}, the authors provide a way to approximately control the state inside the $j$ and $(j+1)$ level sets by using adiabatic deformations of the parameters $\alpha$ and $\beta$. The conical intersections are generic patterns and we can use this type of idea to realize the exchange of energy between different levels, even in the proof of our Theorem \ref{th_main}. However, it is noteworthy that \eqref{conical} cannot hold for rectangular shapes with $\alpha$ and $\beta$ being the size of the rectangle: due to the homothetic invariance, the degeneracy $\lambda_{j}=\lambda_{j+1}$ remains true in a direction $(\alpha,\beta)$. In other words, the family of rectangles behaves as a one-parameter family of domains from the point of view of crossing of the eigenvalues. It means that, similarly to all the previous proposed strategies, we have to seek for conical intersections by slightly breaking the symmetry of the rectangle to obtain a more generic shape and one parameter has to deform the shape away from the family of rectangular shapes. Then the problem of checking \eqref{conical} is equivalent to computations as those of Table \ref{table1}.

\vspace{5mm}

\noindent{\bf \underline{Controls on }q\underline{uasi-rectan}g\underline{ular domains}.}\\[1mm]
Every result presented above is not only guaranteed for the rectangles, but also for domains which are very close to it in the meaning of Theorem \ref{th_cv_spectrum}. In this situation, the spectral behavior of the Hamiltonian generating the dynamics is very close to the one on a rectangle and then, all the techniques above are still valid, up to a small error depending on the domain. 

To apply one of these control processes to a general domain, we can proceed as follows. We can deform the domain adiabatically back-and-forth to an almost rectangular one by preserving the simplicity of the spectrum. While the domain is quasi-rectangular, we apply the prescribed control (see Figure \ref{ex4}).  

\begin{figure}[ht]
\begin{center}
\includegraphics[width=\textwidth-50pt]{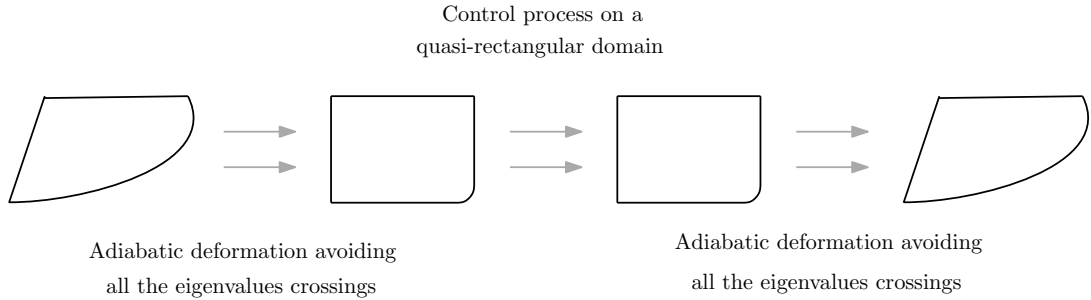}
\end{center}
\caption{\it The figure represents how to apply to the case of general domains the controls discussed before. The first and the last motions are back-and-forth adiabatic  deformations of the domain in a quasi-rectangular one. In the intermediate step instead, we apply the chosen control process.} \label{ex4} 
\end{figure}

\subsection{Pumping motion and Fermi acceleration}\label{section_pumping}
In this section, we discuss in more details the pumping motion introduced above and represented in Figure \ref{ex2}. 

Consider a free quantum-mechanical particle in a rectangular box $[0,a]\times [0,1]$. The eigenfunctions and their corresponding energies are respectively 
$$\phi_{m,n}= \sin\big(\pi m\frac xa\big)\sin\big(\pi ny\big)~~\text{ and }~~E(m,n)= \pi^2 \big(\frac {m^2}{a^2} + n^2\big)$$
where $m,n\geq 1$. We slowly change the length of the rectangular box until reaching the shape $[0,\tilde a]\times [0,1]$, where the eigenfunctions and energies are now 
$$\tilde \phi_{m,n}= \sin\big(\pi m\frac x{\tilde a}\big)\sin\big(\pi ny\big)~~\text{ and }~~\tilde E(m,n)= \pi^2 \big(\frac {m^2}{\tilde a^2} + n^2\big)$$
During the deformation, in the adiabatic limit, the quantum numbers $m$ and $n$ are preserved, i.e., if the process is slow enough, then by starting with the state $\phi_{m,n}$, the system will find itself close to the state $\tilde \phi_{m,n}$. Next, we return the box to its original size and, on the way back, we slowly deform the boundary of the box in such manner that it has a generic non-rectangular shape all the time, except at the start and at the end of the process. The genericity means no crossing of the energy levels, i.e., the instantaneous energy spectrum stays simple, see Section \ref{section_break}. As already noticed, this break of symmetry can also be performed by keeping the rectangular shape and adding a generic non-symmetric potential, which could be simpler in practice. Assume that $a$ and $\tilde a$ are irrational, so that both corresponding spectra in the boxes are simple. Thus, following the total ordering according to the increase of energy, we can define the number $\tilde k(m,n)$ such that $\tilde \phi_{m,n}$ is the $\tilde k(m,n)$-th eigenfunction of the Laplacian in $[0,\tilde a]\times [0,1]$. This number is adiabatically preserved when we return the box to its original shape, meaning that the quantum state finishes close to $\phi_{m',n'}$ such that $k(m',n')=\tilde k(m,n)$ (where $k(m',n')$ is similarly defined by the fact that $\phi_{m',n'}$ is the $k(m',n')$-th eigenfunction of the Laplacian operator in $[0,a]\times [0,1]$).

We obtain a cyclic process such that at its first stage the quantum numbers $m$ and $n$ are preserved, while at the second stage the ordering $k$ of the energy is preserved. Hence the values of $m$ and $n$ at the end of the cycle do not need to be the same as at the beginning. It generates a permutation $\sigma:\NN\rightarrow \NN$ defined by 
$$\sigma(k)=\tilde k(m(k),n(k))$$
with the obvious (abusive) notations that the numbers $m(k)$ and $n(k)$ are uniquely determined by $k(m(k),n(k))=k$. 

\vspace{3mm}

It would be interesting to understand the dynamics of the iterations of this permutation. Indeed, if the pumping motion is slow enough, the quantum state will be successively transformed in the $\sigma^j(k)$-th eigenmode, $j=1,2,3,\ldots$ for as many cycles as we want. In simpler one-dimensional models, a similar permutation process can be explicitly studied, see \cite{one-D,Dmitry}. Here, the existence of two quantum numbers $m$ and $n$ make the rigorous study much more involved. However, we conjecture that typically, $\sigma^j(k)$ grows exponentially with the number of iterations, i.e. the corresponding physical process should exhibit an exponential energy growth. 

There can only be two type of dynamics for bijections of the set of natural numbers: the orbit is either periodic (looped), or it tends to infinity at forward and backward iterations. This depends on initial conditions; however, we conjecture that for a generic choice of $a$ and $\tilde a$, only a small set of initial conditions produce looped orbits. Moreover, for a
typical non-looped orbit, we have
$$ \liminf_{j\rightarrow +\infty } \frac 1j \ln(\sigma^j(k)) > 0~.$$
Physical justification to this claim is given by the second law of thermodynamics: the entropy cannot decrease. Indeed, by definition, the Gibbs volume entropy is the logarithm of the number of states below the given energy level, i.e., it is 
$\ln k$ (see also \cite{HHD,Pereira-Turaev}). We, therefore, expect that the increment of $\ln k$ (the entropy) after each cycle should, typically, be strictly positive. The resulting exponential growth of the energy generated by a periodic motion of a wall is the quantum version of the famous Fermi acceleration.

\vspace{3mm}

To substantiate the growth of entropy claim, we present the following computation. Consider $m,n\geq 1$. In the starting rectangle $[0,a]\times [0,1]$, the numbers $(m',n')\in(\NN^*)^2$ related to an eigenmode $\phi_{m',n'}$ with energy $E(m',n')$ less that $E(m,n)$ are exactly the integer points contained in the ellipse $\{(x,y), \pi^2(x^2/a^2+y^2)\leq E(m,n)\}$. Up to a lower order term, this number is a quarter of the surface of the ellipse and thus  
$$k(m,n)=\frac {\pi a}{4}\big(\frac {m^2}{a^2} + n^2 \big)+o(m^2+n^2)~.$$
In the same way, in the intermediate rectangle $[0,\tilde a]\times [0,1]$, 
$\tilde k(m,n)\sim \frac {\pi \tilde a}{4}\big(\frac {m^2}{\tilde a^2} + n^2 \big)$.
This provides a good estimation for $\tilde k(m,n)$ and thus a good estimation for $\sigma(k)$ if $(m(k),n(k))$ is known. But obtaining $(m(k),n(k))$ from $k$ is very complicated. Thus, we rather consider the mean value of the entropy increase 
$$\delta \Ec(K):=\frac 1K \sum_{k=1}^K \ln(\sigma(k))-\ln(k)$$
for the first $K$ states. Let $E_0:=E(m(K),n(K))$ be the energy of the $K-$th mode in the rectangle $[0,a]\times [0,1]$. The arguments above show that it can be estimated as follows. Let $\Qc=\{(x,y)\in\RR_+^2, \pi^2(x^2/a^2+y^2)\leq E_0\}$ be the quarter of the ellipse corresponding to $E_0$ and $\Qc'=\{(x',y)\in\RR_+^2, \pi^2({x'}^2+y^2)\leq E_0\}$ the quarter of the disk. We have 
\begin{align*}
\delta \Ec(K) &= \frac 1{\text{Vol}(\Qc)} \int_\Qc \ln\big(\frac {\pi \tilde a}4\big(\frac {x^2}{\tilde a^2} + y^2 \big) \big)\,-\,\ln\big(\frac {\pi  a}4\big(\frac {x^2}{a^2} + y^2 \big)  \big)\dd x\dd y~+o(1).\\
&= \frac 1{\text{Vol}(\Qc')} \int_{\Qc'} \ln\big(\frac {a}{\tilde a}{x'}^2 + \frac {\tilde a}a y^2 \big)\,-\,\ln\big({x'}^2 + y^2 \big)\dd x'\dd y~+o(1).\\
&= \frac{4\pi}{E_0} \int_0^{\sqrt{E_0}/\pi} r\dd r \, \int_0^{\pi/2} \ln\big(\frac {a}{\tilde a}\cos^2\theta + \frac {\tilde a}a \sin^2\theta \big)\,-\,\ln\big(\cos^2\theta + \sin^2\theta \big) \dd \theta         ~+o(1).\\
&= \frac 2\pi \int_0^{\pi/2} \ln\big(\frac {a}{\tilde a}\cos^2\theta + \frac {\tilde a}a \sin^2\theta \big) \dd \theta ~ +o(1).
\end{align*}
This growth in mean of the entropy can be estimated numerically. In any case, we know that it is positive. Indeed, the strict concavity of the logarithm ensures that 
\begin{align*}
\int_0^{\pi/2} \ln\big(\frac {a}{\tilde a}\cos^2\theta + \frac {\tilde a}a \sin^2\theta \big) \dd \theta
~>~  \int_0^{\pi/2} &\cos^2\theta\ln\frac {a}{\tilde a} + \sin^2\theta \ln \frac {\tilde a}a  \dd \theta\\
&=~ \ln\frac {a}{\tilde a} \int_0^{\pi/2} \cos^2\theta - \sin^2\theta \dd \theta~=~0.
\end{align*}

One should note that the positivity of $\liminf_{K\rightarrow \infty} \delta \Ec(K)$ is not enough to conclude that $\sigma$ generates long growing orbits. For example, if we consider the pumping that alternates the lengths $a$ and $\tilde a=1/a$, then this will only lead to transpositions, i.e., every orbit will loop after the second iteration. Still, we believe that for generic values of $a$ and $\tilde a$, the iterations of $\sigma$ have the same characteristic as the one of a random process with fast decaying correlations. In the random case, the limit of $\delta\Ec$ is the expectation of the increase of the random sequence $i\mapsto \ln(\sigma^i (k))$, and a Chebyshev-like inequality shows that, almost surely, $i\mapsto \sigma^i (k)$ grows exponentially.

We did numerical experiments to test this prediction. For $a=\pi/2$ and $\tilde a=a/3$, we computed the energies of the first  modes and thus built a table of the values of $\sigma(k)$ for all $k\leq 370800$. The computed value of $\delta \Ec(10^5)$ is $0.28713$, whereas the above integral estimation predicts $0.28768$. An illustration of some orbits of the permutation is given in Figure \ref{fig-permutation}. To investigate our conjecture that periodic orbits are very rare and perhaps in finite number, we looked for the periodic orbits starting with $k\leq 10^5$, with period less than $30$ and never growing above the rank $370800$ (the limit of our computed permutation). We only found $9$ periodic orbits, only two being more complicated than transpositions (see Figure \ref{fig-permutation}).

\begin{figure}[ht]
\begin{center}
\includegraphics[width=14cm]{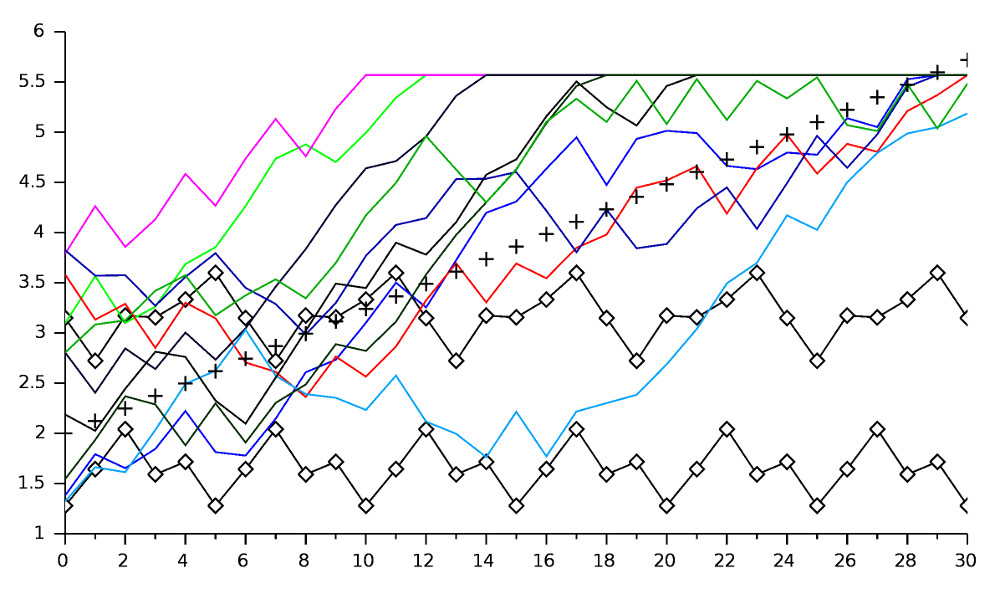}
\end{center}
\caption{\it Some trajectories $j\mapsto \sigma^j(k)$ generated by the permutation associated to the pumping motion with $a=\pi/2$ and the ratio of compression $a/\tilde a=3$. The horizontal axis indicates the ``time'' index $j$ and the vertical one displays $\log_{10}(\sigma^j(k))$. Several randomly chosen trajectories are represented in color (until they reach the bound $370800$, above which our permutation is not computed). Two examples of periodic orbits $(19~ 44~ 110~ 39~   52)$ and $(528~ 1491~ 1429~ 2152~ 3969~ 1407)$ are enhanced with diamonds. The dotted line represents the mean exponential growth rate which is approximatively $0.28$. We notice that the randomly chosen orbits present a variety of growth rates but overall compatible with the mean one.}\label{fig-permutation}
\end{figure}



\end{document}